\documentclass{article}
\usepackage{makeidx}

\usepackage[textsize=tiny]{todonotes}
\usepackage{latexsym,amsfonts,amsthm,amsmath,amscd,amssymb,color}
\usepackage{float}
\usepackage[all]{xy}

\theoremstyle{plain}
\newtheorem{lemma}{Lemma}[section]
\newtheorem{theorem}[lemma]{Theorem}
\newtheorem{proposition}[lemma]{Proposition}
\newtheorem{corollary}[lemma]{Corollary}

\theoremstyle{definition}
\newtheorem{definition}[lemma]{Definition}
\newtheorem{example}[lemma]{Example}
\newtheorem{remark}[lemma]{Remark}
\newtheorem*{definition*}{Definition}

\theoremstyle{remark}

\title{The neighbourhood of a singular leaf}

\author{ Camille Laurent-Gengoux\thanks{Institut Elie Cartan de Lorraine, UMR 7502, Universit\'e de Lorraine,  France},
Leonid Ryvkin \thanks{Institut Math\'ematiques de Jussieu, Universit\'e Paris Diderot, Paris, France.} \thanks{Faculty of mathematics, Universit\"at Duisburg-Essen, Essen, Germany } }

\usepackage{cite}

\usepackage{titling}

\begin{document}

\maketitle

\begin{center}
In memory of Kirill Mackenzie.
\end{center}
\vspace{1em}
\begin{abstract}
An important result for regular foliations is their formal semi-local triviality near simply connected leaves. We extend this result to singular foliations for all $2$-connected leaves and a wide class of $1$-connected leaves by proving a semi-local Levi-Malcev theorem for the semi-simple part of their holonomy Lie algebroid.
\end{abstract}

\setcounter{tocdepth}{2}
\tableofcontents
\vspace{10px}
{\bf Acknowledgements} 
We thank Marco Zambon for a crucial discussion at an early stage of the project, and several precious comments on the final version. We acknowledge valuable comments by Iakovos Androulidakis at relevant places in the text.  Karandeep Singh also suggested several improvements. Both authors are supported by CNRS projet GraNum. L. R. is supported by the PRIME programme of the German Academic Exchange Service with funds from the German Federal Ministry of Education and Research.
\newpage
\section*{Introduction}

Although much less studied than the well-understood regular foliations, singular foliations appear more frequently in differential geometry: orbits of Lie groups actions, symplectic leaves of a Poisson structure, vector fields tangent to an affine variety or annihilating given functions are all instances where the dimension of the leaves may not be constant. All these instances fall into the following category:
\begin{definition*} A singular foliation on a manifold $M$ is a sub-module  $\mathcal F$ of the $C^\infty(M) $-module $\mathfrak X(M) $ of vector fields, which is \emph{(i)} stable under Lie bracket and  \emph{(ii)} locally finitely generated\footnote{
In most of this paper, we will deal with  \emph{locally real analytic} singular foliations $\mathcal F $  (i.e. $M$ is covered by coordinate neighbourhoods in which $\mathcal F$ admits real analytic generators - the change of coordinates does not need to be real analytic), see \cite{LLS}.}.
\end{definition*}

This definition permits to partition $M$ into submanifolds called  \emph{leaves} (Hermann \cite{MR0149402}). There is an open subset of $M$ where $\mathcal F $ is a regular foliation. In particular, at least formally, in a neighbourhood of any such a leaf $L$, the foliation $ \mathcal F$ is entirely described by a group morphism from the fundamental group $\pi_1(L)$ to the group of formal diffeomorphisms of a transversal. In particular, regular foliations are (formally) trivial near simply-connected leaves \cite{MR0055692}.\\

For singular leaves, there have been recent advances in understanding the semi-local structure. Androulidakis and Zambon \cite{AZ13,AZ2} have shown that the holonomy groupoid of $\mathcal F $ (defined previously by Androulidakis and Skandalis \cite{AS}) acts on the normal bundle of the leaf. When the singular foliation is linearizable, this describes the whole semi-local structure. In this article, we mainly focus on the case where $L$ is simply-connected. To our great surprise, we were able to prove that, despite having possibly extremely rich transverse structures, singular foliations remain (formally) trivial near simply-connected leaves, when the transverse singular foliation is made of vector fields vanishing at order at least $2$ (Theorems \ref{theo:transvQuadr} and \ref{thm:TLtoAlinimpliestriviality}). 
When the transverse linear part is not trivial but $L$ is 2-connected, we still have a Levi-Malcev type theorem decomposing $ \mathcal F$ as a semi-direct product of a semi-simple linearizable Lie groupoid action on some transverse singular foliation. The same conclusion holds for simply-connected leaves provided a Levi-Malcev decomposition exists for the linear holonomy Lie algebroid (Theorems \ref{thm:formal} and \ref{thm:formal2}).\\ 

The paper is organised as follows: In Section \ref{sec:holonomy}, we review the notion of holonomy Lie algebroid $A_L$ of a leaf $L$. Using the Artin-Rees theorem, we show that the sub-algebroid of $A_L $ coming from vector fields in $ \mathcal F$ that vanish at least quadratically along $L$ form a nilpotent Lie algebra bundle. This allows us to describe the semi-simple quotient $A_L^s $ of $ A_L$ as a quotient of the linear part $A_L^{lin} $ of $A_L$. 
Using the method of Euler-like vector fields developed in \cite{zbMATH07105909}, 
we show that singular foliations that contain a transverse Euler vector field admit homogeneous generators (see Theorem \ref{theo:EulerExists}). In Section \ref{sec:formalthm}, we state our most central result (Theorem \ref{thm:formal}) and give a geometric reformulation of it (Theorem \ref{thm:formal2}).

Section \ref{sec:applications} applies these results to leaves of dimension $0$, recovering some results of Dominique Cerveau \cite{Cerveau} and deriving consequences for the NMRLA class of \cite{LLS}.  Finally, we show semi-local triviality of transversally quadratic leaves (Theorem \ref{theo:transvQuadr}) and linearly trivial leaves (Theorem \ref{thm:TLtoAlinimpliestriviality}), a phenomenon which is a distinctive feature of singular foliations, with no analogue in the Lie algebroid or Poisson manifold categories (Remark \ref{rmk:nooidnofish}).

\section{Holonomy and connections}
\label{sec:holonomy}
\subsection{The linear holonomy Lie algebroid of a leaf}\label{sec:linhol}

Let $\mathcal F\subset \mathfrak X(M)$ be a singular foliation  on a manifold $M$ (i.e. a  locally finitely-generated $C^{\infty}(M)$-submodule involutive with respect to the Lie bracket).
The singular foliation $\mathcal F $ induces a "singular distribution" defined for every $p \in M $ by:
 $$ T_p \mathcal F :=\left\{~X(p) ~~|~~X \in \mathcal F~\right\}.$$
A fundamental Lemma about singular foliations, originating from Cerveau \cite{Cerveau}, then proved in this context by Dazord \cite{Dazord}, and rediscovered by Androulidakis and Skandalis \cite{AS}, says that singular foliations satisfy a local splitting property, in the following sense: 

 \begin{lemma}[\cite{Cerveau,Dazord,AS}]\label{lem:splitting}
 Let $\mathcal F\subset \mathfrak X(M)$ be a singular foliation on a manifold $M$ of dimension $n$. Every point $p \in M$ admits a neighborhood $U$ on which $\mathcal F $ is isomorphic to the direct product of the following two singular foliations: 
 \begin{enumerate}
     \item  all vector fields on an open ball of dimension $d$, where $ d = {\mathrm{dim}}(T_p \mathcal F)$, and
     \item a singular foliation $ \mathcal T$ on an open ball of dimension $n-d $ made of vector fields that vanish at the origin, and called \emph{transverse singular foliation}. 
 \end{enumerate}
 The germ of the transverse singular foliation does not depend on any choice: any two local isomorphisms as above lead to transverse singular foliations which are locally isomorphic in a neighborhood of the origin. 
 \end{lemma}
 
This splitting lemma is crucial for proving the following results:
\begin{enumerate}
    \item By \cite{MR0149402}, $M$ has a unique decomposition into submanifolds called \emph{leaves} which are the maximal integral subsets for $\mathcal F$). Moreover, the tangent space of the leaf through $p \in M$ is $ T_p \mathcal F$.
    \item Any two points $p_1, p_2$ on the the same leaf admit neighborhoods $U_1, U_2$ on which the singular foliations $\mathcal F|_{U_1}$ and $\mathcal F|_{U_2}$ are isomorphic (\cite{Dazord}). In particular, their germs of transverse singular foliations are isomorphic. It makes sense, therefore, to speak of \emph{the} transverse singular foliation of a given leaf.
\end{enumerate}

In this subsection, we will define several Lie algebroids describing the behaviour of $\mathcal F$ near a chosen leaf $L$.

\begin{definition}[\cite{AS,AZ13}]
Let $\mathcal F$ be a singular foliation and $L$ a locally closed leaf. Let $I_L\subset C^\infty(M)$ be the ideal of functions vanishing along $L$. The \emph{holonomy Lie algebroid} $A_L\to L$ is defined implicitly by the equality $\Gamma(A_L)=\frac{\mathcal F}{I_L\mathcal F}$.
\end{definition}

To verify that this yields a well-defined Lie algebroid, one shows that $\frac{\mathcal F}{I_L\mathcal F}$ is a Lie Rinehart-algebra and a projective $C^\infty(L)$-module (\cite{AZ13}). It is therefore a Lie algebroid. By construction, this Lie algebroid is transitive. It is therefore locally trivial. We denote by $\mathfrak g_L=ker(\rho\colon A_L\to TL)$ its isotropy bundle of Lie algebras. At a given point $p \in L$, $\mathfrak g_p$ is by construction the Androulidakis-Skandalis isotropy Lie algebra (see \cite{AS}) of $\mathcal F$ at $p$.\\

The holonomy Lie algebroid $A_L$ acts on the normal bundle $\nu=\frac{TM|_L}{TL}$ of $L$ in $M$, see \cite{AZ13}. Algebraically, this action can be seen as follows: The space $\Gamma(\nu)$ is isomorphic to $\frac{\mathfrak X(M)}{\mathfrak X_L(M)}$, where $\mathfrak X_L(M)\subset \mathfrak X(M)$ are the vector fields tangent to $L$, i.e. the vector fields $X$ satisfying $X(I_L)\subset I_L$. Now, $\mathcal F\subset \mathfrak{X}_L(M)$ acts on this quotient and $I_L\mathcal F$ acts trivially. The induced action of $\frac{\mathcal F}{I_L\mathcal F}$ on $\frac{\mathfrak X(M)}{\mathfrak X_L(M)}$ is a Lie algebroid action of $A_L$ on the normal vector bundle $\nu$.

This action is equivalently given by a homomorphism of transitive Lie algebroids $A_L\to CDO(\nu)$, where $CDO(\nu)$ is the Lie algebroid of \emph{covariant differential operators} of $ \nu \to L$ introduced 
by Mackenzie (\cite{Mackenzie}). Recall that this Lie algebroid fits into the exact sequence
\[
0\to \mathfrak {gl}(\nu)\to CDO(\nu)\to TL\to 0 
\]
and that its sections can be interpreted as fiberwise linear vector fields on $\nu$.

\begin{example}
For $L$ a regular leaf, $ A_L=TL $ and the $TL$-action on the normal bundle is the Bott-connection.
\end{example}

\begin{definition}
We call \emph{linear holonomy Lie algebroid of $L$} the image Lie algebroid of $A_L\to CDO(\nu)$. We denote this Lie  algebroid by $A_L^{lin}$.
\end{definition}

\begin{example} Let the leaf $L=\{p\}$ be a point. Then $A_L=\frac{\mathcal F}{I_p\mathcal F}=\mathfrak g_p$ is a Lie algebra and $\nu=T_p M$. The linear holonomy Lie algebroid is the Lie subalgebra of ${\mathfrak{gl}}(T_pM) $ obtained by linearizing  all the vector fields in $\mathcal F$:
$$ \xymatrix{\mathcal F \ar[rrr]^{lin(X)} \ar[d]  & & & {\mathfrak{gl}}(T_pM)\\   \mathfrak g_p \ar@{.>}[urrr]_{}  & &  &}.$$
The dotted arrow is well-defined, as vector fields in $I_p\mathcal F$ vanish quadratically at $p$. Its image is the Lie algebra considered in \cite{Cerveau}.
\end{example}

By definition of $ A_L^{lin}$, there is a surjective Lie algebroid morphism $A_L\to A_L^{lin}$. 
Let us  understand its kernel, which is a locally trivial bundle of Lie algebras by transitivity of $ A_L$. 
Since all vector fields in $\mathcal F$ are tangent to the leaf $L$, derivation w.r.t $X \in \mathcal F$ preserves the filtration:
 $$ C^\infty(M)  \supset  I_L \supset I_L^2  \supset  ...$$
 This implies that for all $i,j \geq 0$:
  $$  [\mathcal F \cap I_L^i \mathfrak X(M) , \mathcal F \cap I_L^j \mathfrak X(M)  ] \subset \mathcal F \cap I_L^{i+j-1}  \mathfrak X(M) .$$

This induces a natural filtration on $\Gamma(A_L)$ by $\Gamma(A_L)^i := \frac{\mathcal F\cap I_L^i\mathfrak X(M)}{I_L\mathcal F\cap I_L^i\mathfrak X(M)}$. In words, $\Gamma(A_L)^i$ is ``the space of sections in $A_L$ that can be represented by a vector field in $ \mathcal F$  that vanishes at order $i$ along $L$''.
This filtration obviously satisfies $[\Gamma(A_L)^i, \Gamma(A_L)^j]\subset \Gamma(A_L)^{i+j-1}$ and $[\Gamma(A_L)^i, \Gamma(A_L)]\subset \Gamma(A_L)^{i}$. By construction, we have 

\begin{lemma} \label{lem:filtration} There exists a vector bundle filtration of $A_L$:
$$A_L=  A_L^{0} \supset  A_L^{1}  \supset   A_L^{2}  \supset \cdots $$
 such that $\Gamma(A_L^i)=\Gamma(A_L)^i$. Moreover, for every $i > 0$, $A_L^i$  is a Lie algebra bundle, $A_L^1=\mathfrak g_L$ and $A_L^2=ker(A_L\to A_L^{lin})$. In particular $A_L^{lin}=\frac{A_L}{A_L^2}$.
\end{lemma}

\begin{remark}
In general, the above filtration does not need to terminate, i.e. $\bigcap_{i \geq 0} A_L^i$ might not be the zero vector bundle. For instance, consider the foliation on $\mathbb R$ defined by the vector field $ e^{\tfrac{-1}{x^2}} \tfrac{\partial}{\partial x}$. Then $\{ 0 \}$ is a leaf for which $ A_{\{0\}}^i =A_{\{0\}} = \mathbb R$ for all $i \in \mathbb N$, hence $\bigcap_i A_{\{0\}}^i = \mathbb R$. We will see, in the next subsection, that this pathology can not happen for locally real analytic singular foliations.
\end{remark}

\subsection{Nilpotence and the semi-simple holonomy}

For a locally real analytic singular foliation $ \mathcal F$,
upon restriction to a neighbourhood $U$ of a point $p\in M$, we may assume that $\mathcal F|_U$ has real analytic generators $(X_i)_{i=1}^r $ in some local coordinates, and consider the module $\mathcal F^{ra} $ over real analytic functions it generates, and define \emph{the real analytic holonomy Lie algebroid} of the leaf through $p$ (in $U$) by $ \tfrac{\mathcal F^{ra} }{ \mathcal I  \mathcal F^{ra} } $ (where $\mathcal I  $ are real analytic functions vanishing on $L \cap U$ - which is easily seen to be a real analytic subvariety).

\begin{lemma} \label{lem:ra=smooth}
In the above setting, the (smooth) holonomy Lie algebroid and the real analytic holonomy Lie algebroid are isomorphic (as filtered Lie algebras). In equation: $\Gamma(A_L|_{U \cap L})=\tfrac{\mathcal F^{ra}}{\mathcal I \mathcal F^{ra}}$.
\end{lemma}
\begin{proof}
We have to show that the natural filtered Lie algebra morphism $\tfrac{\mathcal F^{ra}}{\mathcal I \mathcal F^{ra}}\to \tfrac{\mathcal F}{I_L \mathcal F} $ induced by the inclusion is an isomorphism. This is a direct consequence of the statement, that smooth functions form a faithfully flat module over real analytic functions (Corollary VI.1.12 in \cite{MR0212575}).
\end{proof}

\subsubsection{The Artin-Rees Lemma and nilpotence}

In this subsection, we show, that for locally real analytic foliations, the kernel of the linearization homomorphism $A_L\to A_L^{lin}$ is a bundle of nilpotent Lie algebras. The proof is based on the following statement of Commutative Algebra:

\begin{theorem} {{\bf Artin-Rees} (\cite{zbMATH03279238})}
\label{thm:ArtinRees}
	Let $\mathcal X$ be a finitely generated module over a Noetherian
	ring $\mathcal C$, $\mathcal I$ be an ideal of $\mathcal C$ and $ \mathcal F\subset \mathcal X $ a submodule. Then there is a positive integer $c$ such that
	\begin{align} \label{eq:ArtinRees} \mathcal I^{n} \mathcal X \cap \mathcal F=\mathcal I^{n-c}((\mathcal I^c \mathcal X)\cap \mathcal F ) \hspace{.5cm} \mathrm{for ~all}~ n\geq c \end{align}
\end{theorem}
The classical formulation of the Artin-Rees lemma is more general, but we stated the form which is most directly applicable to our situation. In fact, we need the following immediate consequence of \eqref{eq:ArtinRees}, applied for $n=c+1$: 
\begin{align} \label{eq:ArtinRees2} \mathcal I^{c+1} \mathcal X \cap \mathcal F \subset \mathcal I \mathcal F \end{align}

For $ \mathcal O, \mathcal I, \mathcal F, \mathcal X$ as in Theorem \ref{thm:ArtinRees}, we call \emph{Artin-Rees bound} of $\mathcal F$ in $\mathcal X$ at $\mathcal I$ the smallest integer that satisfies Condition \eqref{eq:ArtinRees}.

\begin{theorem}\label{thm:arnilpotence}
Let $\mathcal F$ be a locally real analytic singular foliation and $L$ a leaf. Then $A_L^{c+1} =0$ where $c$ is the Artin-Rees bound of $ \mathcal F$ in $ \mathfrak X(M)$ at $I_L$. In particular, the Lie algebra bundle $A_L^2$ is nilpotent. 
\end{theorem}
\begin{proof}
Let $p\in L$. Upon restriction to a neighbourhood $U$ of $p\in M$, we may assume that $\mathcal F|_U$ has real analytic generators $(X_i)_{i=1}^r $ in some local coordinates. By Lemma \ref{lem:ra=smooth}, we may use the real analytic holonomy Lie algebroid. Artin-Rees Theorem \ref{thm:ArtinRees} applied in the following context: 
\begin{itemize}
    \item $\mathcal C $ is the algebra of real analytic functions on $U$ (which is Noetherian by Theorem III.3.8 in \cite{MR0212575}),
    \item the ideal $\mathcal I$ of real analytic functions on $U$ vanishing along $L$,
    \item  the $\mathcal C $-module $ \mathcal X $ of real analytic vector fields on $U$,
    \item the sub-module  $\mathcal F^{ra} \subset  \mathcal X $ generated be the vector fields $(X_i)_{i=1}^r $,
 \end{itemize}
there exists $ c \in \mathbb N$ that satisfies \eqref{eq:ArtinRees}
and therefore \eqref{eq:ArtinRees2}. 
Geometrically \eqref{eq:ArtinRees2} applied to $n=c+1$ implies that a vector field in $ \mathcal F^{ra}$ 
that vanishes at order $c+1$ along $L$ belongs to $\mathcal I \mathcal F^{ra} $. In terms of the filtration in Lemma \ref{lem:filtration}, it means that  $A_L^{c+1} =0$. In particular $ A_L^2$ is a bundle of nilpotent Lie algebras of depth less of equal to $  c+1 $. 
\end{proof}
The following example (inspired by Grabowska and Grabowski \cite{GG}) illustrates, that the Artin-Rees bound can be arbitrarily large in our situation, i.e. that the filtration on $A_L$ may have arbitrarily many non-zero terms.

\begin{example}
On $M={\mathbb R}^n $, let us give to the coordinates $(x_1, \dots, x_n)  $ the weights $(1, \dots, n) $. Real analytic functions on $M$ then become a filtered algebra. 
Real analytic vector fields that preserve this filtration form a module $\mathcal F $ stable under Lie bracket and generated by the finite family
$$   \left\{ ~  x_1^{i_1}  \dots x_n^{i_n} \tfrac{\partial }{ \partial x_k } ~\middle| ~  k \in [1 : n], ~ i_1\hbox{ }  \dots,i_k \in [ 0:n] \hbox{ and }  i_1 + 2 i_2 + \cdots + n i_n  \geq k ~ \right\}. $$
Since all vector fields in $\mathcal F $ vanish at the origin $0$, $L=\{0\}$ is a leaf. The vector field $ x_1^n \tfrac{\partial}{\partial x_n} $ is an element in $ \mathcal F \cap I^n_L \mathfrak X(M) $, but does not belong to $ I_L \mathcal F$. This implies that the Artin-Rees bound $c$ is greater or equal to $n$.
Since $ I^i_L \mathfrak X(M)  \subset \mathcal F $ for $i \geq n$, we have $ I^{n+1}_L \cap \mathfrak X(M)  \subset \mathcal F  =  I_L \left(I^{n}_L \cap \mathfrak X(M) \right)$, so that $c=n$. 
\end{example}

 The Artin-Rees bound also bounds the possible degrees of generators of a singular foliation which is preserved by some Euler vector field, as stated below.

\begin{proposition}
\label{prop:homogeneous}
Let $ \mathcal F$ be a real analytic singular foliation in a neighborhood of $0$ in $\mathbb R^n $ made of vector fields that vanish at $0$. If $ \mathcal F$ is preserved by the Euler vector field $\mathcal E=\sum_{i=1}^n x_i\tfrac{\partial}{\partial x_i}$, then: 
\begin{enumerate}
    \item Every homogeneous component of a vector field in $ \mathcal F $ belongs to $ \mathcal F$.
    \item $ \mathcal F$ admits homogeneous generators whose degrees are less or equal to the Artin-Rees bound of $\mathcal F$ at $0$.
\end{enumerate}
\end{proposition}

Notice that we will extend Proposition \ref{prop:homogeneous}  to neighbourhoud of leaves (see Theorem \ref{theo:EulerExists} below). 
To prove Proposition \ref{prop:homogeneous}, we start with a lemma:

\begin{lemma}
\label{lem:adminusk}
For $k\geq 1$, the operator  
\begin{align*}
 P^k:\mathfrak{X}(\mathbb R^n)&\to \mathfrak{X}(\mathbb R^n)\\
        X                     &\mapsto ([{\mathcal E},X ]- (k{-}1)\cdot X)
\end{align*}
restricts to an invertible isomorphism of $ \mathcal F\cap I^{k+1}\mathfrak X(\mathbb R^n)$, where $I=I_{\{0\}}$ is the ideal of functions vanishing at $0$.
\end{lemma}
\begin{proof}
The operator $P^k$ restricts to an endomorphism of $\mathcal F\cap I^{k+1}\mathfrak X(\mathbb R^n)$, since  the identity map and the Lie bracket with $\mathcal E$ preserve both $\mathcal F$ and $I^{k+1}\mathfrak X(\mathbb R^n)$. Furthermore, $P^k|_{\mathcal F\cap I^{k+1}\mathfrak X(\mathbb R^n)}$ is injective as the kernel of $P^k$ is given by homogeneous vector fields of degree $k$, a space in trivial intersection with $I^{k+1}\mathfrak X(\mathbb R^n)$.\\

We claim that an explicit inverse to $P^k|_{I^{k+1}\mathfrak X(\mathbb R^n)}$ is given by
\begin{align}\label{eq:Yexplicit}
    Q^k: ~~~~&I^{k+1}\mathfrak X(\mathbb R^n)~~~~~~~~~~~~\to  ~~~~~I^{k+1}\mathfrak X(\mathbb R^n)\\ \nonumber
    &Z  = \sum_{i=1}^n z_i(x) \tfrac{\partial}{\partial x_i}~~~~\mapsto ~~~~~ Q^k(Z)=\int_{t=0}^1  \left( \frac{1}{t^{k+1}} \sum_{i=1}^n z_i(tx) \tfrac{\partial}{\partial x_i}  \right) dt 
\end{align}
The convergence of the integral is granted by the fact that all functions $t \mapsto z_i(tx) $ vanish at order $k+1$ at $0$.
A simple integration by part gives $P^k\circ Q^k=\mathrm{id}_{I^{k+1}\mathfrak X(\mathbb R^n)}$.  It is also clear that each one of the functions $ \int_{t=0}^1 \frac{1}{t^{k+1}} z_i(tx)dt$ belongs to $ I^{k+1}$, so that the image of $Q^k$ is indeed in $I^{k+1}\mathfrak X(\mathbb R^n)$.\\

To conclude the proof, we need to show that $Q^k$ preserves the subspace $  \mathcal F\cap I^{k+1}\mathfrak X(\mathbb R^n)$. First, let us interpret Equation \eqref{eq:Yexplicit} as:
 \begin{align}
\label{eq:YNonexplicit} Q^k(Z) = \int_{t=0}^1  \frac{1}{t^{k}} \mu^t _* (Z)   dt  
\end{align}
where $\mu^t$ is the homothety $x \mapsto \tfrac{x}{t} $.  Since $ \mu^t$ is the the flow at time $-{\mathrm{ln}}(t)$ of $ \mathcal E$, it preserves $\mathcal F $ by Proposition 1.6 in \cite{AS}. In particular, if  $Z \in \mathcal F$, then $\mu^{t}_* (X) \in \mathcal F$ for all $t \in ]0,1]$. Now, since $ \mathcal F$ admits real analytic generators, $ \mathcal F$ is closed with respect to the Fr\'echet topology (see Theorem 2 in Tougeron \cite{MR240826} - the result is attributed to Malgrange). In particular, $ \mathcal F$ is stable under the integration \eqref{eq:YNonexplicit}, so that if $Z\in \mathcal F\cap I^{k+1}\mathfrak X(\mathbb R^n)$, then $ Q^k(Z) \in \mathcal F$. This proves the lemma. 
\end{proof}

\begin{proof} (of Proposition \ref{prop:homogeneous}).
Let us decompose $X \in \mathcal F\cap  I^{k}\mathfrak X(\mathbb R^n)$ as
 $X =  X^{(k)} + R $, with $X^{(k)}$ homogeneous of degree $ k$, and $R \in I^{k+1} \mathfrak X(\mathbb R^n)$. As $ X^{(k)}$ is in the kernel of $P^k$, we have 
 $$ P^k(R)=P^k(X)\in \mathcal F\cap I^{k+1}\mathfrak X(\mathbb R^n).$$
By Lemma \ref{lem:adminusk}, this implies that $R \in  \mathcal F\cap I^{k+1} \mathfrak X(\mathbb R^n) $, so that $ X^{(k)} = X - R \in \mathcal F$. This proves that the lowest component of an element in $ \mathcal F$ is in $\mathcal F $. The first item of the proposition follows by an immediate finite induction.
  
  For $ (e_1, \dots, e_b)$ a local trivialization of $ A_L^{k} / A_L^{k+1} $, let us choose $ (X_1, \dots, X_b )$  a $b$-tuple of elements in $ \mathcal F $ that represent it. The $b$-tuple $ (X_1^{k}, \dots, X_b^{k})$ of their homogeneous components  of degree $k$ is again made of element of $ \mathcal F$ by the first item, and still represents $ (e_1, \dots, e_b)$. Applying this procedure for all $k =0, \dots, r$,
  we obtain a basis of $ \mathfrak g=A_L=A_L^0$ which are all represented by homogeneous vector fields in $ \mathcal F$ of degree less than the Artin-Rees bound.
  In view of Proposition 1.5 item a in \cite{AS}, these vector fields are generators of $ \mathcal F$. This proves the second item.
\end{proof}

\begin{remark}
In the case $\mathfrak g_{0}= \mathfrak g_{0}^{lin} $, Proposition \ref{prop:homogeneous} reduces to Theorem 8.1 in Dominique Cerveau \cite{Cerveau} - a result extended to a neighborhood of a leaf by Marco Zambon \cite{ZambonPrivate}.
\end{remark}

\subsubsection{The Levi exact sequence and the semi-simple holonomy Lie algebroid} 

\label{sec:semisection}

For a Lie algebra $\mathfrak g $, the Levi-Malcev decomposition theorem goes as follows: \emph{(i)}  $ \mathfrak g$ has a unique maximal solvable ideal $\mathfrak{rad}(\mathfrak g) $,  \emph{(ii)} the quotient  $ \mathfrak g /\mathfrak{rad}(\mathfrak g) $ is a semi-simple Lie algebra $  \mathfrak g^s $, and  \emph{(iii)} there is a section $  \mathfrak g^s  \hookrightarrow  \mathfrak g $. 

The Levi-Malcev theorem does not easily generalize to transitive Lie algebroid. Remark \ref{rmk:Levi(iii)} gives a counter-example to step \emph{(iii)}, but steps \emph{(i)} and \emph{(ii)} admit generalizations that we now describe: For $A\to L$ a transitive Lie algebroid with anchor $\rho$ the isotropy Lie algebra bundle $\mathfrak g := {\mathrm ker}(\rho) $ is locally trivial (\cite[Theorem 8.2.1]{Mackenzie}), so that the \emph{fiberwise radical}, i.e. the disjoint union $ \coprod_{l \in L}  \mathfrak{rad} (\mathfrak g_l) $ is indeed a Lie algebra bundle over $L$. We denote it by ${\mathfrak{rad}}(A) $. Since $A$ is, near every point $m \in L $, a direct product of $TL \to L$ with its isotropy Lie algebra at $m $ (see, e.g. Theorem 1.2 in \cite{Zung}),  sections of  ${\mathfrak{rad}}(A) $ form an ideal of $ \Gamma(A)$. This proves the following proposition:

\begin{proposition}
For every transitive Lie algebroid $A$ over $L$, 
the quotient $ A/\mathfrak{rad}(A) $ is a transitive Lie algebroid over $L$, with semi-simple isotropies, and  
\begin{align}\label{eq:shortLevi}
\xymatrix{\mathfrak{rad}(A)\ar[r]&A\ar[r]&{A} / {\mathfrak{rad}(A)}}
\end{align}
is a short exact sequence of Lie algebroids.
\end{proposition}
%\begin{proof}
%It suffices to show that sections of  ${\mathfrak{rad}}(A) $ form an ideal of $ \Gamma(A)$.

%We prove it as follows: every $a \in \Gamma(A)$ induces a family $\Phi_t^a : A \to A $ of Lie algebroid diffeomorphisms by:
%\begin{align}\label{eq:defFlow}\Phi_0^a  (b)= b \hbox{ and }  \frac{\partial \Phi_t^a (b)}{\partial t}  = [ a , \Phi_t^a (b) ] \hspace{.3in} \hbox{ $ \forall b \in \Gamma(A) $}. \end{align}
%For all values of $t$ such that the flow of $ \rho(a)$ is defined, the induced endomorphism $\Phi_t^a \colon \Gamma(A)  \to \Gamma (A)$ comes from a vector bundle map $ A \to A$ (over the flow of $\rho(a)$) which is easily checked to be a Lie algebroid isomorphism\footnote{Geometrically, it amounts to act on $A$ through the adjoint action of the bisection obtained by integrating the left-invariant vector field $\overrightarrow{a}$.}. In particular, $\Phi_t^a  $ preserves  $\Gamma({\mathfrak{rad}}(A)) $, so that evaluating  \eqref{eq:defFlow}
%at $t=0$ for some $b \in \Gamma( {\mathfrak{rad}}(A)) $, we obtain that 
%$$ [a,b] =\left. \frac{\partial \Phi_t^a (b)}{\partial t} \right|_{t=0} \in  \Gamma({\mathfrak{rad}}(A)).$$
%
%This implies that the quotient $ A/\mathfrak{rad}(A) $ comes with a natural transitive Lie algebroid structure. By construction, its isotropy Lie algebras are semi-simple, and the exact sequence \eqref{eq:shortLevi} lies in the Lie algebroid category.
%\end{proof}

These general considerations lead to the following definition

\begin{definition}
Let $L$ be a leaf of a singular foliation $\mathcal F $.
We call the quotient $\frac{A_L}{\mathfrak{rad}(\mathfrak g_L)}$ the \emph{semi-simple holonomy Lie algebroid} of $L$
and denote it by $ A^s_L$.
\end{definition}

\begin{remark}
\label{rmk:Levi(iii)}
It is not true in general that the short exact sequence \eqref{eq:shortLevi} admits a Lie algebroid section, even for holonomy Lie algebroid $A_L$ of a leaves of a singular foliation.
For instance, for $ A =TM \oplus \mathbb R$ equipped with the Lie algebroid structure associated to a closed $2$-form $ \omega \in \Omega^2(M)$, the semi-simple holonomy Lie algebroid  $A/{\mathfrak{rad}}(A)$ is the tangent Lie algebroid $TM$, but a Lie algebroid section $TM \hookrightarrow A $ exists if and only if $ \omega$ is exact. 
\end{remark}

\begin{proposition}
\label{prop:linTos}
Let $L$ be a leaf of a locally real analytic singular foliation $\mathcal F $.
There is a natural Lie algebroid morphism $\xymatrix{A_L^{lin} \ar@{.>}[r]& A_L^s} $ that makes the following diagram commutative:
 $$ \xymatrix{ A_L \ar[r]\ar[rd] & A_L^{lin}\ar@{.>}[d] \\ & A_L^s}   $$
\end{proposition}
\begin{proof}
By Theorem \ref{thm:arnilpotence}, the kernel of $A_L \to A_L^{lin} $ is a bundle of nilpotent Lie algebras. It is therefore contained in ${{\mathfrak{rad}}(A_L) } $, i.e. in the kernel of the natural projection $A_L \to A_L^s $. 
The result follows.
\end{proof}
\subsection{Connection theory}
\subsubsection{Ehresmann or Levi $\mathcal F$-connections and flatness}

According to Proposition \ref{prop:linTos}, for every leaf $L$ of a locally  singular foliation $ \mathcal F$, we have the following sequence of surjective morphisms of transitive Lie algebroids over $L$: 
\begin{align*}
\xymatrix{
\mathcal F\ar@{-->}[rr]&& A_L\ar[rr]&& A_L^{lin} \ar[rr] &&A_L^{s} \ar[rr]&& TL
},
\end{align*}  
where the leftmost arrow is dashed, as it is not a morphism of Lie algebroids. 
The main purpose of this article is to describe the behaviour of $\mathcal F$ in a neighborhood of $L$, using the semi-simple holonomy Lie algebroid $ A_L^s$. 

To start, we will consider as in \cite{LGR} neighbourhoods $U$ of $L$ in $M$ which are small enough in the following sense: they have to admit a projection $\pi \colon U \to L $ such that $ T_x \mathcal F + {\mathrm{ker}}(T_x \pi) = T_x M $ for all $x\in U$.
 These pairs $(U,\pi) $ shall be called \emph{$ \mathcal F$-neighbourhoods} and satisfy several important properties, in particular, by Proposition 2.21 in \cite{LGR}:

\begin{lemma} 
Every locally closed leaf $L$ of a singular foliation $\mathcal F$ admits a $ \mathcal F$-neighbourhood $ (U,\pi)$.
\end{lemma}

As we are only interested in the behaviour of $\mathcal F$ near $L$, for the rest of the section we assume that $M=U$ is an $\mathcal F $-neighborhood equipped with some projection $\pi:M\to L$. The $C^\infty(L)$-modules of \emph{$\pi$-vertical} (resp. \emph{ $\pi$-projectable}) vector fields will be denoted by $\mathfrak X^v\subset \mathfrak X^{proj}\subset\mathfrak X(U)$. We also write $\mathcal F^v$ and $\mathcal F^{proj}$ for the $\pi$-vertical and $\pi$-projectable vector fields in $\mathcal F$.

\begin{lemma}
  The Lie algebra $\mathcal F^{proj}$ of $\pi$-projectable vector fields in $\mathcal F$ form a Lie-Rinehart algebra over $C^\infty(L)$. Moreover, there is a sequence of surjective Lie-Rinehart algebra morphisms:
\begin{align}\label{xypic:surjections}
\xymatrix{
\mathcal F^{proj}\ar[rr]&& \Gamma( A_L)\ar[rr]&&  \Gamma(A_L^{lin}) \ar[rr] && \Gamma(A_L^{s} )\ar[rr]&& \mathfrak X(L)
},
\end{align} 
\end{lemma}

We introduce several types of connection adapted to this context:

\begin{definition}[\cite{AndroulidakisPrivate,LGR}]\label{def:connection}
Let $L$ be a leaf of a foliation $ \mathcal F$ on $(M=U,\pi)$.
\begin{itemize}
    \item An \emph{Ehresmann $\mathcal F$-connection} is a $C^\infty(L)$-linear section $\mathfrak{X}(L)\to \mathcal F^{proj}$ of the surjection $\mathcal F^{proj}\to\mathfrak X(L)$ in \eqref{xypic:surjections}.
    \item A \emph{Levi $\mathcal F$-connection} is a $C^\infty(L)$-linear section $s:\Gamma(A_L^s)\to \mathcal F^{proj}$  of the surjection $\mathcal F^{proj}\to \Gamma(A_L^s)$ in \eqref{xypic:surjections}.
\end{itemize}
\end{definition}

Existence of Ehresmann $\mathcal F$-connection was already established in \cite{AndroulidakisPrivate,LGR}. We now extend this result:

\begin{proposition}\label{prop:exist_connexion} 
Let $L$ be a leaf of a foliation $ \mathcal F$ on $(M=U,\pi)$. Then, possibly on a sub-$\mathcal F$-neighbourhood of $L$, Levi $ \mathcal F$-connections and Ehresmann $ \mathcal F$-connections exist.
\end{proposition}
The proposition follows from the following lemma:
\begin{lemma}\label{lem:existlift}
Let $L$ be a leaf of a foliation $ \mathcal F$ on $(M=U,\pi)$. Then, possibly on a sub-$\mathcal F$-neighbourhood of $L$ the surjection $\mathcal F^{proj}\to \Gamma(A_L)$ admits a $C^\infty(L)$-linear section.
\end{lemma}
\begin{proof}
When $ L$ is a point, $A_L $ is a Lie algebra, and such a section $s $ exists. In particular, there exists for all $p \in L $ a linear section $s_p : {\mathfrak g}_p \to \mathcal F_p^v $. In view of the splitting Lemma \ref{lem:splitting}, every point $p$ admits a neighborhood $U$ in $L$ which admits an neighborhood $V$ on $M$ such that 
$$ A_L |_U  \simeq TL |_U  \oplus  {\mathfrak g}_p \hbox{ and } \mathcal F^{proj}|_V \simeq \Gamma(TL ) \oplus \mathcal F^v.$$
Under this identification, $ ({\mathrm id} \times s_p)$ is a section on $\mathcal F^{proj}|_V \to \Gamma(A_L )|_U $.

Let $ (U_i)_{i \in I}  $ be an open cover of $L$, such that each $ U_i$ comes equipped with a $C^\infty(U_i)$-linear section  $ s_i$ of $ p \colon \mathcal F^{proj}|_{V_i} \to \Gamma(A_L)|_{U_i} $ for some open subset $V_i \subset p^{-1}(U_i)$. Without any loss of generality, one can assume that the open cover  $ (U_i)_{i \in I}  $  is locally finite and comes with a partition of unit $(\chi_i)_{i \in I} $. Then:
 $$ V := \bigcup_{x \in L} \left(\bigcap_{i \in I \hbox{ s.t. } x \in U_i}  V_i \right) $$
 is an open neighborhood of $L$, and $ s := \sum_{i \in I} \chi_i s_i $ is a well-defined  $C^\infty(L)$-linear section  of $ p \colon \mathcal F^{proj}|_{V} \to \Gamma(A_L) $
\end{proof}
\begin{proof} (of Proposition \ref{prop:exist_connexion})
The composition of any linear section $TL\to A_L$ (resp.  $A_L^s\to A_L$) with a section as in Lemma \ref{lem:existlift} yields an Ehresmann $\mathcal F$-connection (resp. a Levi $\mathcal F$-connection). This proves the statement.
\end{proof}

\begin{definition}\label{def:flat}
An Ehresmann/Levi $\mathcal F$-connection is called \emph{flat} if it is bracket-preserving (i.e. a morphism of Lie-Rinehart algebras). 
\end{definition}

Let us give the geometric interpretation of the existence of flat Ehremann/Levi connections:
\begin{proposition} Let $L$ be a leaf of a singular foliation $\mathcal F $.
\begin{itemize}
    \item A flat Ehresmann $\mathcal F$-connection exists if and only if near $L$ there exists a regular foliation included into $\mathcal F $ admitting $L$ as a leaf. 
 \item
A flat Levi $\mathcal F$-connection exists if and only if there exists a Lie algebroid action of $ A_L^s$ on $\pi\colon M \to L$ made of vector fields in $ \mathcal F$. 

\end{itemize}
\end{proposition}

\begin{remark}
Let $L$ be a locally closed leaf. 
 A flat Ehresmann connection induces a flat section of the anchor map $A_L^s\to TL$.
Also, if $L$ is Ehresmann-flat, then its normal bundle $ \nu$ is a flat bundle. This gives clear obstructions to the existence of flat Ehresmann-connections. In contrast, flat Levi $\mathcal F$-connection can be assured to exist under relatively mild topological conditions, as we will see later.
\end{remark}

\subsubsection{Linear Ehresmann or Levi $\mathcal F$-connections}

An additional desirable property for a Ehresmann or Levi $\mathcal{F}$-connection is (transverse) linearity. For this purpose, we need to notion of fiberwise linearity. This is completed through the following definition adapted from \cite{zbMATH07105909}:

\begin{definition}
Consider a $ \mathcal F$-neighborhood  $ (U,\pi)$ of a locally closed leaf $L$. A vector field $\mathcal E \in U$ which is:
\emph{(i)} tangent to the fibers of $\pi \colon U \to L$,
\emph{(ii)}  vanishes along $L$,
\emph{(iii)}  whose linearization is the Euler vector field on the normal bundle $ \nu$, and \emph{(iv)} that is complete
 is said to be an Euler-like vector field on  $ (U,\pi)$. 
\end{definition} 
Upon rescaling the vector field and shrinking the tubular neighbourhood $(U,\pi) $ the completeness condition \emph{(iv)} can always be assumed for  vector field satisfying \emph{(i)}--\emph{(iii)}. The following lemma is the adaptation of \cite{zbMATH07105909} to the case where $ \pi$ is given and $\mathcal E $ is assumed to be tangent to it.

\begin{lemma}\label{lem:Meirenken}
 Euler-like vector fields on an $\mathcal F $-neighborhood $(U,\pi) $ are in one-to-one correspondence with vector bundle structures on the fiber bundle $\pi:U\to L$.
\end{lemma}

\begin{remark}
Every such a vector bundle is isomorphic to the normal bundle $\nu =\tfrac{TM|_L} {TL}$ of $L$ in $M$.
\end{remark}
    
\begin{definition}
Consider a $\mathcal F $-neighborhood $(U,\pi) $ of a locally closed leaf $L$.  We say that a Levi (resp. Ehresmann) $\mathcal F$-connection $ s \colon \Gamma(A_L^s)  \hookrightarrow \mathcal F^{proj}_U$  (resp.  $   s \colon \Gamma(TL)  \hookrightarrow \mathcal F^{proj}_U  $)
  is \emph{linear} with respect to an Euler-like vector field $\mathcal E$ if there exists a neighborhood of the zero section on which every vector field in the image of $s$ commutes with $\mathcal E$.
\end{definition}

\begin{remark}
Upon identifying $U$ with a vector bundle as in Lemma \ref{lem:Meirenken} vector fields commuting with $\mathcal E $ are simply fiberwise linear vector fields.
\end{remark}

We say that a vector field $X$ is homogeneous of degree $k$ with respect to $ \mathcal E$ is $[\mathcal E,X ] = (k-1) X $. (Linear vector fields are then homogeneous of degree $ 1$). Upon choosing adapted coordinates $(x,y) $ where $\mathcal E = \sum_{i=1}^d x_i \tfrac{\partial}{\partial x_i} $, homogeneous vector fields of degree $k$ are vector fields of the form:
 $$  \sum_{ i} f_i(x,y) \frac{\partial}{\partial x_i} + \sum_{ j} g_j(x,y) \frac{\partial}{\partial y_j}$$
 where $ x \mapsto  f_i(x,y)$ and $ x \mapsto  g_j(x,y)$ are homogeneous polynomials of degree $ k$ and $k-1$ respectively for every value $y$.

\begin{theorem}
\label{theo:EulerExists}
Let $\mathcal F$ be a locally real analytic singular foliation with leaf $L$.  If $\mathcal F$ is preserved by an Euler-like vector field $\mathcal E$ along $L$, then near $L$:
\begin{enumerate}
    \item Any homogeneous component of a vector field in $ \mathcal F$ is in $\mathcal F $.
    \item The foliation $\mathcal F$ is generated by homogeneous vector fields (of degree less or equal than the Artin-Rees bound of the transverse foliation).
    \item There exists a $C^\infty(L)$-linear section $s:\Gamma(A_L^{lin})\to \mathcal F^{proj}$ preserving the Lie bracket.
\end{enumerate}
\end{theorem}

\begin{proof}
Fixing a tubular neighbourhood adapted to $\mathcal E$, the expressions for $P^k$ and $Q^k$ from Lemma \ref{lem:adminusk} still make sense, and satisfy the same relations on  vector fields tangent to $L$. Hence, we can proceed identically to Proposition \ref{prop:homogeneous}. This proves the first two items. A  bracket-preserving isomorphism from  $\Gamma(A_L^{lin})$
to  linear vector fields in $ \mathcal F$ is obtained by mapping  $X \in \Gamma(A_L^{lin})$ to the linear component of any of its inverse image in $\mathcal F $. 
\end{proof}

For foliations as in Theorem \ref{theo:EulerExists}, the existence of a flat section $s:\Gamma(A_L^{s})\to \mathcal F^{proj}$ is therefore equivalent to the existence of a Lie algebroid section $s:A_L^s \to A_L^{lin}$.

\begin{corollary}
Let $\mathcal F$ be a locally real analytic singular foliation with leaf $L$. If $\mathcal F$ is preserved by an Euler-like vector field $\mathcal E$ along $L$, then a flat Levi $\mathcal F $-connection exists if and only if a Lie algebroid section $ A_L^s \to A_L^{lin} $ exists.
\end{corollary}

\section{The formal neighbourhood of a simply connected (singular) leaf}
\label{sec:formalthm}

\subsection{Lie algebroid cohomology of small degrees}

In order to prove our central Theorem \ref{thm:formal}, we will need the following statements about Lie algebroid cohomology. The Whitehead Lemmas I and II admit generalizations for Lie algebroids that we now state, using several results of Mackenzie (in \cite{MR2157566}):

\begin{lemma}[Whitehead Lemma I for Lie algebroids]\label{lem:h1iszero} 
Let $A\to L$ be a transitive Lie algebroid with semi-simple isotropies $\mathfrak g_L=ker(\rho)$.
If $\pi_1(L)=0$, then the Lie algebroid cohomology group $H^1(A,E)$ is trivial for any flat finite-dimensional $A$-module $E \to L$.
\end{lemma}
\begin{proof}
Theorem 7.4.5 in \cite{MR2157566} asserts that there is a spectral sequence whose first page is $H^t(L, H^s_{CE}(\mathfrak g_L, E))$ converging to $H^\bullet(A, E)$. The only two terms contributing to $H^1(A,E)$ are $H^0(L,H^1_{CE}(\mathfrak g_L, E))$ and \newline $H^1(L,H^0_{CE}(\mathfrak g_L, E))$. The former space is trivial, due to Whitehead Lemma for Lie algebras and the semi-simplicity of $\mathfrak g_L$. The second one is trivial, because $L$ is simply connected.
\end{proof}

\begin{lemma} [Whitehead Lemma II for Lie algebroids]\label{lem:WhiteheadII}
Let $A\to L$ be a transitive Lie algebroid with semi-simple isotropies $\mathfrak g_L=ker(\rho)$. If $\pi_1(L)=\pi_2(L)=0$, then the Lie algebroid cohomology group $H^2(A,E)$ is trivial for any flat finite-dimensional $A$-module $E\to L$.
\end{lemma}
\begin{proof}
The proof uses the same spectral sequence as the proof of Lemma \ref{lem:h1iszero}. Here, the components of the first page required to be trivial are $H^0(L, H^2_{CE}(\mathfrak g_L, E))$, $H^1(L, H^1_{CE}(\mathfrak g_L, E))$, and $H^2(L, H^0_{CE}(\mathfrak g_L, E))$. The former two are again trivial by the Whitehead Lemmas for Lie algebras, and the last one by 2-connectedness of $L$ (in view of Theorem 6.5.16 in \cite{MR2157566}).
\end{proof}

We prove the following generalization of the Levi-Malcev theorem, which is also a prototype for our central Theorem \ref{thm:formal}.
\begin{proposition}
\label{LeviMalcev}
Let $A$ be a transitive Lie algebroid over a 2-connected base $L$ and $A^s$ its semi-simple quotient. Then there exists a Lie algebroid section $s:A^s\to A$ of the projection $ \pi \colon A \to A^s$. 
\end{proposition}
\begin{proof}
The Lie algebra bundle $ \mathfrak{rad}(A) = {\mathrm{ker}}(\pi)$, defined in Section \ref{sec:semisection} comes with a terminating natural filtration by Lie algebra bundles
\begin{align} \label{eq:filt} \mathfrak r^0=  \mathfrak{rad}(A) ~~~\supset~~~ \mathfrak r^1=[\mathfrak r^0, \mathfrak r^0]~~~\supset~~~ \mathfrak r^2=[\mathfrak r^1,\mathfrak r^1]~~~~\supset~~~ ...~~~\supset \mathfrak r^N=0\end{align}
such that the subquotients $\tfrac{\mathfrak r^i}{\mathfrak r^{i+1}}$ are Abelian.
 We construct $s$  by induction.
 Let $s^0: A^s \to A$ be any linear section: Its curvature is valued in the radical $\mathfrak r^0=\mathfrak{rad}(A)$.
Assume that there exists a section $ s^i : A^s \to A$ whose curvature $ c^i$ is
$\mathfrak r^i$-valued, then:
\begin{enumerate}
    \item the quotient space $  \mathfrak r^{i} / \mathfrak r^{i+1} $ is an $ A^s$-module for: 
$ (\xi,\overline{b} )  \mapsto \overline{ [ s^i (\xi) , b ]}$,
with $ \xi \in \Gamma(A^s), b \in \Gamma(\mathfrak r^i)$,
    \item the skew-symmetric bilinear map:
$$  (\xi, \zeta) \mapsto  c^i(\xi, \zeta) \mod \mathfrak r^{i+1} $$
is a Lie algebroid $2$-cocycle of $ A^s $, valued in $  \mathfrak r^{i} / \mathfrak r^{i+1} $.
\end{enumerate}
By the second Whitehead Lemma \ref{lem:WhiteheadII}, this cocycle is a coboundary, so that there exists $\sigma_i : A \to \mathfrak r^i  $ such that:
  $$  c_i (\xi,\zeta)  = [s^i (\xi), \sigma^i (\zeta) ]  + [\sigma^i (\xi) , s^i (\zeta)] - \sigma^i ([\xi,\zeta]) \hbox{ mod }  \mathfrak r^{i+1} .$$ 
This means that the curvature of $s^{i+1} = s^i + \sigma^i $ is  
 $\mathfrak r^{i+1}$-valued. Since the filtation \eqref{eq:filt} terminates at degree~$N $, $ s^N$ is a Lie algebroid section.
\end{proof}
 
 \begin{remark}
 Proposition \ref{LeviMalcev} should not be confused from Zung and Monnier-Zung's Levi Theorem for Lie algebroids \cite{Zung, ZungMonnier} which is valid for any Lie algebroid, not only transitive ones, but is a local result.
 \end{remark}

Let us state an immediate consequence of Proposition \ref{LeviMalcev} and Theorem \ref{thm:arnilpotence}: 

\begin{corollary}\label{cor:sectohol}
Let $\mathcal F$ be a singular foliation and $L$ a locally closed leaf. If $\pi_1(L)=\pi_2(L)=0$, then there exists a Lie algebroid section $A_L^s\to A_L^{lin}$. In case $\mathcal F$ is locally real analytic, then there even exists a Lie algebroid section $A_L^s\to A_L$.
\end{corollary}

\subsection{Formal singular foliations}

In this section, we define the formal counterparts of several notions that have been studied in this article.

For $L$ be a submanifold of $M$, we denote by ${\widehat{\mathcal C}} $ the algebra of formal functions along $L$, i.e. (by Borel's Theorem) the quotient of $ C^\infty (M)$ by the ideal of functions vanishing with all their derivatives along $L$.
We call \emph{formal vector fields along $L$} derivations of  $\mathcal {\widehat{\mathcal C}} $.

We call \emph{formal singular foliation along $L$} locally finitely generated  ${\widehat{\mathcal C}} $-submodules of formal vector fields along $L$ which are closed under Lie bracket.
For every singular foliation $ \mathcal F$, the tensor product ${\widehat{ \mathcal F}} := {\widehat{\mathcal C}} \otimes_{ C^\infty(M)} \mathcal F $ is a formal singular foliation along $L$ called the \emph{formal jet of $ \mathcal F$ along $L$}. 
We will only consider formal singular foliations for which $L$ is a leaf, i.e. such that the restriction to $L$ is onto $\mathfrak X(L) $.

As for the non-formal case, we call \emph{holonomy Lie algebroid} of a formal singular foliation $\widehat{\mathcal F} $ along $L$ the Lie algebroid whose space of sections is the quotient ${\widehat{ \mathcal F}}/ \widehat{I}_L {\widehat{ \mathcal F}} $, with  $ \widehat{I}_L = {\widehat{\mathcal C}} \otimes_{ C^\infty(M)}  I_L\subset {\widehat{\mathcal C}}$ the ideal of formal functions vanishing along $L$. As in Lemma \ref{lem:ra=smooth}, since formal functions are a faithfully flat module over real analytic functions (cf.
Theorem III.4.9 in \cite{MR0212575}), we have:

\begin{lemma} If a singular foliation $ \mathcal F$ is locally real analytic, then for every leaf $L$,  the holonomy Lie algebroids of $\mathcal F $ and of its formal jet $ \widehat{\mathcal F}$ along $L$ are isomorphic. In equation: ${\mathcal F}/  I_L {\mathcal F} \simeq \widehat{\mathcal F}/ {\widehat I}_L \widehat{\mathcal F} $.
\end{lemma}

The notions of Levi $\mathcal F $-connections and Ehresmann $\mathcal F $-connections have formal equivalents for a formal singular foliation $ \widehat{ \mathcal F}$ near a leaf $L$.
First, let us choose a tubular neighborhood $p: U \to L $. A formal vector field $X$ is said to be projectable if it preserves $p^*  \mathcal C^\infty(L) \subset \widehat{\mathcal C}$.  Denote formal projectable vector fields in  $\widehat{ \mathcal F} $ by $ \widehat{ \mathcal F}^{proj} $.
By construction, $ \widehat{ \mathcal F}^{proj} $ is a Lie-Rinehart algebra over $C^\infty(L)$.
There are natural Lie-Rinehart algebra morphisms from  $ \widehat{ \mathcal F}^{proj} $ to 
$ \Gamma(A_L^s) $ and $ \mathfrak X (L) $ respectively as in \eqref{xypic:surjections}. 
We call
\emph{formal Levi $\mathcal F $-connection} (resp. \emph{formal Ehresmann $\mathcal F $-connection}) a $C^\infty(L)$-linear section $\Gamma(A_L^s)\to\widehat{ \mathcal F}^{proj}$ (resp $\mathfrak X (L) \to \widehat{ \mathcal F}^{proj}  )$ of these natural surjections.

\begin{remark}
\label{rmk:formalFlow}
For future use, notice that  the time $t$-flow of a formal vector field $X$ tangent to $L$ is a well-defined algebra isomorphism of $\widehat{\mathcal C} $, as long as the time $t$-flow of its restriction to $L$ is defined. If $X$ is $p$-projectable for some $p \colon M \to L $, this formal diffeomorphism maps $p$-fibers to $p$-fibers. 
\end{remark}

\subsection{Existence of flat Levi $ \mathcal F$-connections}

The main goal of the section is to prove the following theorem:

\begin{theorem}\label{thm:formal}
Let $L$ be a locally closed leaf of the locally real analytic foliation $\mathcal F$. If $\pi_1(L)=0$, and there exists a Lie algebroid section $z$ from the semi-simple holonomy $A_L^s$ to the linear holonomy $A_L^{lin}$, then there exist
\begin{itemize}
    \item an $\mathcal F$-neighbourhood $(U,\pi)$,
    \item a formal Levi  $ \mathcal F$-connection $s^\infty:\Gamma(A_L^s)\to \widehat{ \mathcal F}^{proj}$,
    \item and a $\pi$-vertical formal Euler-like vector field $\mathcal E^\infty$.
\end{itemize}
such that
\begin{itemize}
\item $s^\infty$ is linear with respect to $\mathcal E^\infty$, i.e. $[s^\infty(\xi), \mathcal E^\infty]=0$ for all $\xi\in \Gamma(A_L^s)$,
\item  $s^\infty$ is flat, i.e.   $[s^\infty(\xi), s^\infty(\zeta)]=s^\infty([\xi,\zeta])$ for all $\xi,\zeta \in \Gamma(A_L^s)$.
\end{itemize}
\end{theorem}

Here is an immediate consequence of this Theorem and Corollary \ref{cor:sectohol}.

\begin{corollary}
\label{thm:formalpi2}
Let $L$ be a locally closed leaf of the locally real analytic foliation $\mathcal F$. If $\pi_1(L)=\pi_2(L)=0$, then the conclusions of Theorem \ref{thm:formal} hold.
\end{corollary}

Theorem \ref{thm:formal} will be proved by induction. The initial case is given by Lemma \ref{lem:initial}. Proposition \ref{prop:step} gives the induction step.
Both result depend on the technical Lemma \ref{lem:euler}.
 
We denote by $ \mathfrak X^v$ vertical vector fields for $\pi \colon U \to L$ so that, for every $ k \in \mathbb N$, $ I_L^k\mathfrak X^v$ stands for vertical vector fields vanishing at order $k$ along~$L$.

\begin{lemma}\label{lem:euler}
Let  $\mathcal E$ be a $\pi$-vertical Euler-like vector field on $(U,\pi)$. For every $\pi$-vertical vector field $X \in \mathfrak X^v$ and avery $ k \geq 2$:
\begin{enumerate}
    \item  if $X \in  I_L^{k}  \mathfrak X^v$, then $ \tfrac{1}{k-1}[\mathcal E,X ] -  X \in  I_L^{k+1}  \mathfrak X^v$;
    \item  if the linearization of $X$ along $L$ is zero, and 
      $ [\mathcal E , X ] \in  I_L^{k}  \mathfrak X^v $, then $ X   \in  I_L^{k}  \mathfrak X^v$.
\end{enumerate}
\end{lemma}
\begin{proof}
It suffices to check both items in local adapted coordinates $(x,y) $ where $y = (y_1, \dots, y_d)$ are local coordinates on $L$ and $ x= (x_1, \dots, x_s)$ are local coordinates on the fibers of $\pi :  (x,y) \mapsto y  $ such that
 $ \mathcal E = \sum_{i=1}^s x_i \tfrac{\partial}{\partial x_i} $. 
 Since  the ideal $ I_L $ is generated by $ x_1, \ldots, x_s $, the Taylor expansion implies that for every  $X \in  I_L^{k}  \mathfrak X^v$ there exists functions $ f_{i;i_1, \dots , i_s }(y) $ such that
  $$ X   =  \sum_{i=1}^s \sum_{i_1+\dots+i_s =k  } f_{i;i_1, \dots , i_s }(y) x_1^{i_1} \dots x_s^{i_s}     \tfrac{\partial}{\partial x_i}  \,  \hbox{~~~~mod~~} \,   I_L^{k+1} \mathfrak X^v.$$
 Since  $[\mathcal E , I_L^{k+1} \mathfrak X^v ] \subset  I_L^{k+1} \mathfrak X^v $, the first item follows from the easily checked identity:
  \begin{align}
      \label{eq:Taylor} [\mathcal E ,x_1^{i_1} \dots x_s^{i_s}     \tfrac{\partial}{\partial x_i}  ] \, =  \, (k-1) \, x_1^{i_1} \dots x_s^{i_s}     \tfrac{\partial}{\partial x_i}   .
  \end{align}

Let us prove  the second item. By assumption, the linearization of $X \in \mathfrak X^v$ along $L$ is zero, so that $ X \in I_L^2 \mathfrak X^v$. 
The conclusion then follows from considering the Taylor expansion of $X$, in view of Equation \eqref{eq:Taylor}.
\end{proof}

A pair $ (s,\mathcal E) $ with $s$ a Levi $\mathcal F $-connection and $ \mathcal E$ a $\pi$-vertical Euler-like vector field are said to be \emph{flat and linear up to order $k$} if:
\begin{itemize}
\item $[s(\xi), \mathcal E]=0 \mod  I^k_L\mathfrak X^v$ for all $\xi\in \Gamma(A_L^s)$,
\item  $[s (\xi), s(\zeta)]=s([\xi,\zeta]) \mod  I^k_L\mathfrak X^v$ for all $\xi,\zeta \in \Gamma(A_L^s)$.
\end{itemize}

\begin{lemma}\label{lem:initial} For every Euler-like vector field $\mathcal E$ and every Levi $\mathcal F$-connection $s$:
\begin{itemize}
    \item $[\mathcal E,s(\xi)]=0\mod I^2\mathfrak{X}^v $,
    \item If, in addition, $s$ projects to a Lie algebroid morphism $z \colon A_L^s\to A_L^{lin}$, then  $[s (\xi), s(\zeta)]=s([\xi,\zeta]) \mod  I^2_L\mathfrak X^v$ for all $\xi,\zeta \in \Gamma(A_L^s)$.
\end{itemize}

\end{lemma}
\begin{proof}
Since $  \mathcal E$  is an Euler-like vector field, a neghborhood $U$ of $L$ can be covered by local charts, equipped with coordinates $ (x_1, \dots, x_s,y_1, \dots, y_d)$, such that:
\begin{enumerate}
    \item  $L$ is locally given by $x=0$, and $p$ is locally given by $ (x,y) \mapsto y$,
    \item  $\mathcal E  $ coincides with the vector field $\sum_{i=1}^s x_i\tfrac{\partial}{\partial x_i}.$
\end{enumerate}
The Taylor expansion of any $p$-projectable vector field $X$ is given by 
$$X= \sum_{i=0}^d f_i(y) \tfrac{\partial}{\partial y_i}  + \sum_{i,j=1}^s  a_{ij}(y)x_i\tfrac{\partial}{\partial x_j}+I_L^2\mathfrak X^v.$$
A simple calculation gives $ [\mathcal E,X] \in I_L^2 \mathfrak X^v$, as $\sum_{i=0}^d f_i(y) \tfrac{\partial}{\partial y_j}  + \sum_{i,j=1}^s  a_{ij}(y)x_i\tfrac{\partial}{\partial x_j}$ commutes with $\mathcal E$ and $I_L^2\mathfrak X^v$ is preserved by $\mathcal E$ by Lemma \ref{lem:euler}. Since $ s(\xi)$ is $p$-projectable for all $ \xi \in \Gamma(A_L^s)$, the  first assertion follows.  The second assertion is a consequence of the following facts:
\begin{enumerate}
    \item the assumption on $s$ means that for all $\xi, \zeta \in \Gamma(A_L^s)  $, the  $\pi$-vertical vector field $[s(\xi),s(\zeta)]-s([\xi,\zeta])$ is contained in the kernel of $\mathcal F^{proj}\to \Gamma(A_L^{lin})$,
    \item the kernel of $\mathcal F^{proj}\to \Gamma(A_L^{lin})$ is contained in $I_L^2\mathfrak X^v = \mathfrak X^v\cap I^2_L\mathfrak X^{proj}$ by definition of $A_L^{lin}$.
\end{enumerate}
This completes the proof of the second argument.
\end{proof}

\begin{proposition}\label{prop:step}Let $(s^k,\mathcal E^k) $ be a Levi $\mathcal F$-connection and an Euler-like vector field respectively, which are flat and linear up to order $k$,  with $k \geq 2$.
Then there exists  $(s^{k+1}, \mathcal E^{k+1}) $,   a Levi $\mathcal F$-connection and an Euler-like vector field respectively, such that 
\begin{itemize}
\item the pair $(s^{k+1}, \mathcal E^{k+1}) $ is flat and linear up to order $k+1$,
\item the vector fields $ \mathcal E^k$ and $  \mathcal E^{k+1}$ coincide up to order $k$,
\item the vector fields  $s^{k}(\xi)$ and $s^{k+1}(\xi)$ coincide up to order $k$ for all $\xi \in \Gamma(A_L^s)$.
\end{itemize}
\end{proposition}
\begin{proof} By assumption, the pair $(s^k,\mathcal E^k)$ satisfies two conditions: ``Linearity up to order $k$'' and ``Flatness up to order $k$'', i.e. for all $\xi,\zeta \in \Gamma(A_L^s)$
\begin{align}
[s^k(\xi), \mathcal E^k]=0 \mod  I^k_L\mathfrak X^v \tag{Lin\textsuperscript{k}}\label{link}\\
[s^k (\xi), s^k(\zeta)]=s^k([\xi,\zeta]) \mod  I^k_L\mathfrak X^v \tag{Flat\textsuperscript{k}} \label{flatk}.    
\end{align}

We will prove the proposition in three steps:\\
{\bf Step 1.} We construct a class that measures the failure of \eqref{link} to hold at order $k+1$. 
\begin{itemize}
        \item The $C^\infty(L)$ module $\frac{I^k_L\mathfrak X^v}{\mathcal F \cap I^k_L\mathfrak X^v + I^{k+1}\mathfrak X^v}$ is projective, i.e. isomorphic to the section space $\Gamma(V^k)$ of a vector bundle $V^k\to L$. Indeed, the quotient $\frac{I^k_L\mathfrak X^v}{I^{k+1}_L\mathfrak X^v}$ is given by sections of some vector bundle over $L$: It is a direct consequence of the splitting Lemma \ref{lem:splitting}, that $V^k$ is a quotient of that bundle.
        \item A Lie algebroid action of $A_L^s$ on $V^k$ is defined by $\nabla_\xi (\overline{\sigma}):= \overline{[s^k(\xi), \sigma]}$, with $\sigma\in {I^k_L\mathfrak X^v}$. The action is well-defined, because $[s^k(\xi),\cdot ]$ preserves $\mathfrak X^v$,  $I_L$ and $\mathcal F$, hence it preserves the numerator and denominator of $\frac{I^k_L\mathfrak X^v}{\mathcal F \cap I^k_L\mathfrak X^v + I^{k+1}_L\mathfrak X^v}$. As $k\geq 2$, Equation \eqref{flatk} implies that the action $\nabla $ is flat.
        \item By \eqref{link}, there is a well-defined $C^\infty(L)$-linear map $def:\Gamma(A_L^s)\to \Gamma(V^k)$ given by $\xi\mapsto \overline{[s^k(\xi), \mathcal E^k]}$ describing the defect of linearity up to order $k+1$. The curvature 
\begin{align}
\label{curvatureC}
        c^k(\xi,\zeta)=[s^k (\xi), s^k(\zeta)]-s^k([\xi,\zeta])
\end{align}
        is valued in $\mathcal F$ and in $I_L^k\mathfrak X^v$, by assumption \eqref{flatk}. The first item in Lemma \ref{lem:euler} implies that applying $[\cdot ,\mathcal E^k]$ to a vector field in $\mathcal F\cap I_L^k\mathfrak X^v$ yields an element of $\mathcal F \cap I^k_L\mathfrak X^v + I^{k+1}\mathfrak X^v$: Therefore upon applying $[\cdot,\mathcal E^k]$ to Equation \eqref{curvatureC}, we obtain 
        \[
       [s^k (\xi), [s^k(\zeta),\mathcal E^k]]- [s^k(\zeta), [s^k (\xi),\mathcal E^k]]= [s^k([\xi,\zeta]),\mathcal E^k] \mod \mathcal F \cap I^k_L\mathfrak X^v + I^{k+1}\mathfrak X^v,
        \]
    which is exactly the cocycle condition: $\nabla_\xi def (\zeta) - \nabla_\zeta def (\xi) - def([\zeta,\xi]) $.
    \end{itemize}  
{\bf Step 2.}   We construct $(\mathcal E^{k+1},s^{k+1})$ satisfying (\ref{link}\textsuperscript{\color{red}+1}).
\begin{itemize}
    \item Since $L$ is simply connected, Lemma \ref{lem:h1iszero} implies that the class $[def]\in H^1(A_L^s,V^k)$ is zero. Choose $\varepsilon^k\in I_L^k\mathfrak X^v$, such that $\overline{\varepsilon^k}$ is a primitive of $def$:
    \[ 
    { def(\xi) = \nabla_\xi \epsilon^{k} \hbox{ i.e. } [s^k(\xi), \mathcal E^k]}=[s^k(\xi),\varepsilon^k] \mod  \mathcal F \cap I^k_L\mathfrak X^v + I^{k+1}_L\mathfrak X^v .
    \] We define the new Euler-like vector field by $\mathcal E^{k+1}=\mathcal E^k-\varepsilon^k$.
    \item By construction of $\mathcal E^{k+1}$, for every given $\xi\in \Gamma(A_L^s)$, there exists a vector field $\sigma^k(\xi)$ in $\mathcal F\cap I_L^k\mathfrak X^v$, such that
    \[
        {[s^k(\xi), \mathcal E^{k+1}]}=\sigma^k(\xi) \mod  I^{k+1}_L\mathfrak X^v.
    \]
    Using local trivializations and partitions of unity on $L$, the map $\xi\mapsto \sigma^k(\xi)$ can be achieved to be $C^\infty(L)$-linear. We now define $s^{k+1}=s^k - \tfrac{\sigma^k}{k-1}$. By construction, $s^{k+1}$ is still a Levi $\mathcal F$-connection.
    \item Let us verify (\ref{link}\textsuperscript{\color{red}+1}):
    \[
    [s^{k+1}(\xi),\mathcal E^{k+1}]=   [s^k(\xi) - \tfrac{\sigma^k(\xi)}{k-1},\mathcal E^{k+1}]=\sigma^k(\xi)-\frac{1}{k-1}[\sigma^k(\xi),\mathcal E^{k+1}]=0 \mod I_L^{k+1}\mathfrak X^v.
    \]
    The last equality holds by the first item of Lemma \ref{lem:euler}.
    \end{itemize}
 {\bf Step 3.}  Consider the curvature \[c^{k+1}(\xi,\zeta)=[s^{k+1} (\xi), s^{k+1}(\zeta)]-s^{k+1}([\xi,\zeta]).\]We show that  (\ref{link}\textsuperscript{\color{red}+1}) implies  (\ref{flatk}\textsuperscript{\color{red}+1}), i.e. $c^{k+1}=0$ modulo $I_L^{k+1}\mathfrak X^v$.
 \begin{itemize}
     \item  Since $s^{k+1}=s^k$ modulo $I_L^k\mathfrak X^v$, we know that
     \[
     c^{k+1}=c^k=0 \mod I_L^k\mathfrak X^v.
     \]
     \item In view of (\ref{link}\textsuperscript{\color{red}+1}), all underbraced terms in the following expression are in $I_L^{k+1}\mathfrak X^v$.
     \begin{align*}
         &[c^{k+1}(\xi,\zeta),\mathcal E^{k+1}] 
         &=[s^{k+1} (\xi), \underbrace{[s^{k+1}(\zeta),\mathcal E^{k+1}]}]- [s^{k+1}(\zeta), \underbrace{[s^{k+1} (\xi),\mathcal E^{k+1}]}]- \underbrace{[s^{k+1}([\xi,\zeta]),\mathcal E^{k+1}] }.
     \end{align*}
    Since $s^{k+1}(\xi)$ and $s^{k+1}(\zeta)$ are projectable vector fields tangent to $L$, their bracket with $ I_L^k\mathfrak X^v$ takes values in  $I_L^k\mathfrak X^v$.  Hence, $[c^{k+1},\mathcal E^{k+1}]=0\mod I_L^{k+1}\mathfrak X^v$. 
     \item The second item of Lemma \ref{lem:euler}, implies that $c^{k+1}=0\mod I_L^{k+1}\mathfrak X^v$, i.e.  (\ref{flatk}\textsuperscript{\color{red}+1}) holds.
 \end{itemize}
    
 This completes the proof.
\end{proof}

\subsection{Examples and counter-examples}

Let us give counter examples to naive generalizations of Theorem \ref{thm:formal}. Let us explore the non-simply connected case.

\begin{example}
For $L$ a leaf in a regular foliation, we have $A_L = A_L^{lin}=A_L^s =TL $, and every tubular neighbourhood $(U,\pi) $ induces a unique flat Levi $ \mathcal F$-connection: it suffices to lift a vector field in $L$ to the unique $ \pi$-projectable vector field in $ \mathcal F$. However, the transverse formal Euler-like vector field can only exist if the holonomy  $\Phi(\gamma) $ is a formally linearizable diffeomorphism of the transversal for all $\gamma \in \pi_1(L) $. The regular foliations (with dimension $1$ leaves) obtained by  suspension of diffeomorphism $ \phi\colon \mathbb R^n \to \mathbb R^n$ fixing $0$ are instances of such foliations with $L \simeq S^1$ if $ \phi$ is not formally linearizable at zero (e.g. $n=2$ and $(x,y) \mapsto (x,y+x^2) $). 
\end{example}

\begin{example}\label{ex:snake}
Consider the ``self-eating snake'' singular foliation, as in Figure 1, realized as follows. Let $ \mathcal S$ be the ``foliation by concentric circles'', i.e. the singular foliation on $ \mathbb R^2$ of all vector fields $X$ such that $X[ \phi]=0$, with $\phi = \sum_{i=1}^2 x_i^2$. Then consider the direct product singular foliations on $ \mathbb R^2 \times \mathbb R$ given by
 $ \mathcal F  := \mathcal S \times \mathfrak X(\mathbb R) $.
This foliation goes to the quotient through the equivalence relation
 $ (x,t) \sim ( \tfrac{1}{2}x, t+1) $, for all $(x,t) \in \mathbb R^n \times  \mathbb R$. The only singular leaf of the quotient singular foliation is $ L=S^1$. 

In this case, the normal bundle $\nu$ is trivial as a vector bundle, a flat Levi $ \mathcal F$-connection exists (which is also a flat Ehresmann-connection since $A_L^s = TL $), but there is no Ehresmann $ \mathcal F$-connection that makes the normal bundle isomorphic to the trivial one (i.e. the first return map on $\nu $ induced by any Ehresmann $ \mathcal F$-connection is non-trivial).

\begin{figure}[H]\label{fig:snake}
\begin{center}
	\includegraphics[scale=0.5]{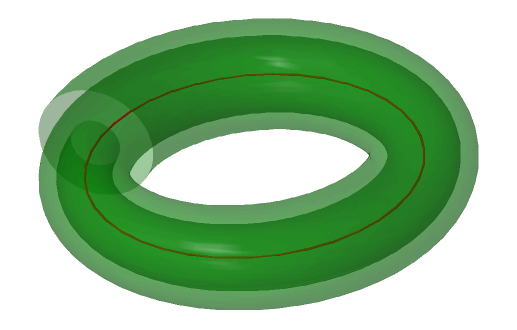}
	\caption{Example \ref{ex:snake}}
\end{center}
\end{figure}

\end{example}

 Here is an example of a leaf  for which a section $A_L^s \to A_L^{lin} $ exists (because it is transversally quadratic and  $A_L^s = A_L^{lin} $) but does not admit a flat Levi $\mathcal F $-connection.

\begin{example}\label{ex:transquad}
Let $L$ be a manifold, $ \alpha, \beta \in \Omega^1 (L)$ closed $1$-forms such that the class of $\alpha \wedge \beta   $ in $ H^2(L)$ is not trivial. On $M := L \times \mathbb R$, we consider for all $a,b,c \in \mathbb N$ with $ 2 \leq a < b $ and  $c = a+b -1 $ the $C^\infty(M) $-submodule  ${\mathcal F} \subset \mathfrak X(M) $ generated by the vector fields: 
  $$  t^c \tfrac{\partial}{\partial t} \hbox{ and } \psi(u) :=  u + \alpha(u)t^a \tfrac{\partial}{\partial t}  +  \beta(u)t^b \tfrac{\partial}{\partial t} \hbox{ with $u \in \mathfrak X(L) $ } .$$
 A direct computation shows that:
 \begin{align*}
     &[t^c \tfrac{\partial}{\partial t} , \psi (u) ] =  \left((c-a) \alpha(u)t^{a-1} + (c-b) \beta(u) t^{b-1} \right) \, t^c \tfrac{\partial}{\partial t} \\
    &\psi([u,v]) - [\psi(u) ,\psi (v) ] = (b-a)   \left(  \alpha(u) \beta (v) - \beta (u) \alpha (v)  \right) \, t^c \tfrac{\partial}{\partial t},
 \end{align*}
 so that ${\mathcal F} $ is a singular foliation. 
 
 By construction, $L \times \{0\}$ is a leaf of ${\mathcal F} $ and $I_L$ is then the ideal generated by $t$. The computations above imply that the holonomy Lie algebroid is $ A_L= TL \oplus  \mathbb R$, the projection on $TL$ is as anchor map, and the bracket is given for all  $u,v \in \mathfrak X(L), f,g \in C^\infty(L)$ by:
  \begin{align*}
      [ (u , f), (v , g) ] = ([u,v], u[g] - v [f]+ \omega (u,v) )
  \end{align*}
  where $ \omega = (b-a) \, \alpha \wedge \beta  \in \Omega^2 (L)$.
  Any Ehresmann $\mathcal F $-connection is of the form:
   $$ u \mapsto \psi(u) + \gamma_t(u)  t^c \frac{\partial}{\partial t} $$
   for some $t$-dependent $1$-form $\gamma$ on $L$. Its curvature is $ (\omega + d_{dR} \gamma_0) t^c \tfrac{\partial}{\partial t} + o(t^{c}) $.
  Since $\omega$ is not exact, the leaf $L $ does not admit a flat Ehresmann $\mathcal F$-connection.
\end{example}

The following is an example of a possibly simply connected leaf, which does not admit a section $A_L^s\to A_L^{lin}$ and therefore does not admit a flat Levi $\mathcal F $-connection.
\begin{example}
Let $L$ be a manifold and let $ \omega \in H^2(L,\mathbb Z)$ be the Chern class of an $S^1$-bundle $P \to L$ with connection $\theta$. Consider the associated $\mathbb C $-bundle $ \pi \colon E \to L$, and equip it with the linear Ehresmann connection associated to $ \theta$. The horizontal lifts $\hat{u} $ of vector fields $u$ on $L$, together with the infinitesimal vector field $R$ of the $S^1$-action generate a singular foliation $  \mathcal F $ on $E$. For this foliation $ \mathcal F  $ the zero section $L$ is a leaf. In view of the relation:
 $$ [\hat{u} +\pi^*(f)\cdot R, \hat{v}+\pi^*(g)\cdot  R] = \widehat{[u,v] } + \pi^*( u[g] - v [f]+ \omega(u,v))\cdot  R \hbox{~~~~ for all $ u,v \in {\mathfrak X}(L), f,g \in C^\infty(L)$}$$
 the holonomy Lie algebroid of $L$ is $ A_L = TL \oplus \mathbb R$, its anchor map is the projection onto  $TL$, and its Lie bracket is given for all  $u,v \in \mathfrak X(L), f,g \in C^\infty(L)$ by
 $ [ (u ,f), (v , g) ] = ([u,v], u[g] - v [f]+ \omega (u,v) ) $. As $\omega$ is nonzero in cohomology, there can be no Lie algebroid section from $TL=A_L^s$ to $A_L=A_L^{lin}$.
\end{example}
%Let $A \to L$ be a Lie algebroid acting on a vector bundle $V \to L $ and let $\tau \colon \Gamma(A) \to \mathfrak X(V) $ be this action. Let $ \mathcal R$ be a singular foliation on the manifold $V$, which is made of vertical vector fields. 

Let us construct examples for which Theorem \ref{thm:formal} holds. Let $\tau:\Gamma(A)\to\mathfrak X(V)$ be the action of a Lie algebroid $A\to L$ on the vector bundle $V\to L$ and $\mathcal R\subset I_L\mathfrak X(V)$ a singular foliation made of vertical vector fields. If 

\begin{enumerate}
    \item  $\mathcal R $ is invariant under the Lie algebroid action,
    \item $\tau(\Gamma(A))$ intersects $\mathcal R$ trivially, 
\end{enumerate}
 then  vector fields in $\mathcal R $, together with the  vector fields for the infinitesimal   $A$-action on $ V$, generate a singular foliation $\mathcal F $ on $ V$. 
 %Also $ \Gamma(A)\ltimes\mathcal R  \subset  \mathcal F^{proj}.$
The same construction can be completed when $ \mathcal R$ is substituted by a formal singular foliation  $ \widehat{\mathcal R}$ along the zero section $L$. Let us fix notations:
\begin{definition}
\label{def:semiDirectProduct}
Let $A \to L$ be a Lie algebroid acting faithfully on  $ V$. For every (maybe formal) singular foliation  $\mathcal R$  satisfying the above conditions 1. and 2., then the above singular foliation $ \mathcal F$ is called the \emph{semi-direct product} of $A$ with $\mathcal R$  and is denoted\footnote{Notice that it is not true that $ \mathcal F^{proj} \simeq  \Gamma(A)\ltimes\mathcal R$ as Lie algebras (there is only an inclusion).}  by 
$$ \mathcal F :=  A \, \widehat \ltimes \, \mathcal R.$$
 \end{definition}
   
   \begin{example}
   \label{ex:directprod}
   Consider a singular foliation  $\mathcal R $ on $\mathbb R^n $.
   The \emph{direct product} of $L$ with $\mathcal R $ is obtained by choosing, in Definition \ref{def:semiDirectProduct},  $A$ to be $TL$ and  $ V $ to be the trivial $TL$-module $\mathbb R^n \times L$.
   %(on which $TL$ acts trivially), the singular foliation $ \mathcal F$ of Definition \ref{def:semiDirectProduct} is 
   %$C^infty(V)\mathfrak X (L)$ with $\mathcal R $.
   %shall be said to be the \emph{direct product} of $L$ with $ \mathcal R$.
   \end{example}
   
   \begin{example}
   Let $ n \geq 3$.
   Since $ SO (n)  $ acts on the sphere $ S^n$,
   there is a natural action of the transformation Lie algebroid $ A=  {\mathfrak{so}} (n) \times S^n \to S^n$ on $ V= TS^n \to S^n$.
   Let $ \mathcal R $ be the singular foliation on $TS^n$ generated by the Euler vector field. The assumptions in Definition \ref{def:semiDirectProduct} are satisfied. The semi-direct product $ A \widehat \ltimes \mathcal R   $
   is a singular foliation on $ TS^n$, admitting $L=S^n$ as a leaf. For this leaf, the Euler field and the Lie algebroid action above give the formal Euler-like field and the formal Levi $\mathcal F $-connection whose existence is granted by Theorem \ref{thm:formal}.
   
   For $ n=2$, the construction of the singular foliation $\mathcal F $ still makes sense, but $ A_L^s =TL$ and $ A_L^{lin} = {\mathfrak{so}} (3)  \oplus \mathbb R $ (the isotropy of this Lie algebroid at every point in $S^2$ is an Abelian two-dimensional Lie algebra). Since there is no Lie algebroid section $ TS^2 \to {\mathfrak{so}} (3)\oplus \mathbb R $,  Theorem \ref{thm:formal} does not apply.
   \end{example}

\subsection{Geometric reformulation}

Let $L$ be a locally closed leaf of $ \mathcal F$.
Assume both conditions in Theorem \ref{thm:formal} are satisfied: $L$ is simply connected  and a section $z:A_L^s \to A_L^{lin} $ exists. Theorem \ref{thm:formal} then provides:
\begin{enumerate}
    \item[$\spadesuit$] a formal Euler-like vector field $\mathcal E $, tangent to the fibers of $\pi\colon U \to L$.
    \item[$\diamondsuit$] a formal Levi $\mathcal F $-connection $s\colon \Gamma(A_L^s ) \to \widehat{\mathcal F}^{proj} $. 
\end{enumerate}
Moreover, the image of $s$ is made of vector fields commuting with $ \mathcal E$. Let us spell out the content of this data:
\begin{enumerate}
    \item[$\spadesuit$] The formal Euler-like vector field  $\mathcal E $ yields a formal isomorphism $\Phi $ between the fibers of normal bundle $\pi: \nu = \tfrac{TM|_L}{TL} \to L$ to the the fibers of $\pi: U\to L$ that identifies, by construction,  $\mathcal E $ with the Euler vector field $ \mathcal E_\nu$ of the normal bundle. We use $\Phi$ to transport to $ \nu$ the formal jet $\widehat{\mathcal F}$ of $\mathcal F$.
\item[$ \diamondsuit$] The composition $ \Phi^{-1} \circ s$  now becomes a flat Levi $\Phi^{-1}(\widehat{\mathcal F} )$-connection on the fibers of $ \pi \colon \nu \to L$.
\end{enumerate}
Moreover, the image of $\Phi^{-1} \circ s$ is made of vector fields commuting with $ \mathcal E_\nu$, i.e. linear vector fields on $ \nu$, so that the flat Levi $\mathcal F $-connection of item $ \diamondsuit$ is now an Lie algebroid action of $ A_L^s$ on the normal bundle. This proves the following Lemma.

\begin{lemma}
\label{lem:bothActionCoincide}
The image of $\Phi^{-1} \circ s$ is made of linear vector fields on $\nu \to L $. More precisely, for every $ \xi \in \Gamma(A_L^s)$, the linear vector field on $ \nu$ describing the natural  Lie algebroid action of $z(\xi) \in \Gamma(A_L^{lin})$ on $ \nu $ coincides with $ \Phi^{-1} \circ s(\xi)$.
\end{lemma}

We call \emph{radical foliation of $\mathcal F$} the subspace  $\mathcal R \subset \mathcal F^v $ of all vector fields in $\mathcal F^v $ whose image through the linearization map along $L$ is in the radical of $A_L^{lin}$. In view of the definition of $A_L^s $, it can be defined by:
 $$ \mathcal R := {\mathrm{Ker}}  ( \mathcal F^{proj} \to \Gamma( A_L^{s}))  .$$
 \begin{lemma}
\label{lem:radicalFoliation}
The space $ \mathcal R$ is a singular foliation on $M$, included in $ \mathcal F^v$,  and
$$  [s(\Gamma(A_L^s)) , \widehat{\mathcal R}] \subset\widehat{\mathcal R}   {\mathrm{ ~~and~~ }} s(\Gamma(A_L^s)) \oplus \widehat{\mathcal R} = \widehat{\mathcal F}^{proj}.$$
\end{lemma}
\begin{proof}
The formal jet $\widehat{\mathcal R}$ of $\mathcal R$ along $L$ is the kernel of   $\widehat{ \mathcal F}^{proj} \to \Gamma(A_L^s)$, and $s$ is a Lie algebra section of that projection.
\end{proof}

It follows from Lemma \ref{lem:radicalFoliation} that $\widehat{\mathcal F}^{proj}$ is, as a Lie algebra, isomorphic to the semi-direct product:
 $$ \widehat{\mathcal F}^{proj}  \simeq s(\Gamma(A_L^s)) \ltimes \widehat{\mathcal R}. $$
Using the formal diffeomorphism $\Phi^{-1} $, we see that $\widehat{\mathcal F} $ is indeed a singular foliation of the form described in Definition \ref{def:semiDirectProduct} applied to $ A=A_L^s \to L, \nu = V$ and $ \widehat{\mathcal R}_\nu = \Phi^{-1}(\widehat{\mathcal R})$. Using this language Theorem \ref{thm:formal} takes the following form: 

\begin{theorem}\label{thm:formal2}
Let $L$ be a locally closed leaf of the locally real analytic foliation $\mathcal F$ on a manifold $M$. If $\pi_1(L)=0$, and there exists a Lie algebroid section $z$ from the semi-simple holonomy $A_L^s$ to the linear holonomy $A_L^{lin}$, then:
\begin{enumerate}
    \item the normal bundle $ \nu=\tfrac{TM|_L}{TL} \to L $ comes equipped with a flat $A_L^s $-connection,
    \item there is a formal diffeomorphism between $M$ and $ \nu$ (near $L$) that identifies $\widehat{\mathcal F} $ and a semi-direct product\footnote{For the notation $  A_L^s  \widehat{\ltimes} \widehat{\mathcal R}_\nu $, see Definition \ref{def:semiDirectProduct}.} singular foliation on $\nu \to L $ of the form:
     $$  \widehat{\mathcal F} = A_L^s \,   \widehat{\ltimes} \, \widehat{\mathcal R}_\nu$$
     where ${\mathcal R}_\nu $ is a vertical singular foliation on $ \nu$, tangent the fibers of $ \nu \to L$, invariant under the  $A_L^s $-action on $ \nu$, isomorphic to the formal jet of the radical foliation $\mathcal R $ of $\mathcal F$.
\end{enumerate}
\end{theorem}

\begin{remark}
The decomposition $  \widehat{\mathcal F} = A_L^s  \widehat{\ltimes}  \widehat{\mathcal R}_\nu$ must not confuse the reader. Vector fields arising from the infinitesimal action of sections of the Lie algebroid $\Gamma(A_L^s)$ on $ \nu$ are in direct sum with $  \widehat{\mathcal R}_\nu$.
But the $\mathcal C$-module generated by this infinitesimal action is a singular foliation that does in general intersect $ \widehat{ \mathcal R}_\nu$.
The corollary below gives an example where this module contains $  \widehat{\mathcal R}_\nu$.\end{remark}

\begin{corollary}
\label{coro:holonomyIsSemiSimple}
Let $L$ be a locally closed leaf of the locally real analytic foliation $\mathcal F$ on a manifold $M$. If $\pi_1(L)=0$, and $A_L^s=A_L$, then there is a formal diffeomorphism between $\mathcal F $ and the singular foliation associated to the natural Lie algebroid action of the holonomy Lie algebroid $A_L$ on the normal bundle.
\end{corollary}
\begin{proof}
According to Proposition 1.5 in \cite{AS}, in a neighbourhood of $p \in L $, the singular foliation $ \mathcal F$ is generated by any family $ X_1, \dots, X_d$ of vector fields in $ \mathcal F$ whose image in $ \Gamma (A_L)$ is a local trivialization of $ A_L$. As a consequence, the image of $s$ generates $ \mathcal F$.
The result then follows from Theorem \ref{thm:formal2}.
\end{proof}

\section{Local and semi-local structure of a singular foliation}

\label{sec:applications}

\subsection{Local structure of a singular foliation: Levi theorems}
\label{sec:point}

Let us explore the consequences of Theorem \ref{thm:formal} in the neighbourhood of a point $p$ in a manifold $M$ equipped with a singular foliation $\mathcal F $. Splitting Lemma \ref{lem:splitting} allows to make the additional assumption that all vector fields vanish at $p$, upon replacing $M$ with a small disk transversal to the leaf through $m$ if necessary. 

\begin{center}
Throughout Section \ref{sec:point},  $\mathcal F$ shall be a locally real analytic \\ singular foliation made of vector fields that vanish at a point $p \in M$.
\end{center}

\subsubsection{Relation with Dominique Cerveau's Levi theorems}

The requirements of Theorem \ref{thm:formal} (namely  \emph{``If $\pi_1(L)=0$, and there exists a Lie algebroid section $z$ from the semi-simple holonomy $A_L^s$ to the linear holonomy $A_L^{lin}$''}) hold automatically: 
\begin{enumerate}
    \item Since $L$ is reduced to the point $\{p\} $, it is simply connected.
    \item The Lie algebroids $A_L, A_L^{lin}, A_L^s $ are finite dimensional Lie algebras:
    \begin{enumerate}
        \item $A_L$ is the isotropy Lie algebra $\mathfrak g_p$ at $p$,
        \item $A_L^{lin}$ is the quotient $ \mathfrak g_p / \mathfrak g_p^{\geq 2}$
        \item $A_L^{s} $ is the semi-simple part $ \mathfrak g_p^s$ of the Lie algebra $\mathfrak g_p $. 
    \end{enumerate}
    Now, in view of the usual Levi-Malcev decomposition theorem for finite dimensional Lie algebras, a Lie algebra section  $\mathfrak g_p^s \to \mathfrak g_p$ exists. Its composition with the natural projection $ \mathfrak g_p \to \mathfrak g_p^{lin} $ is a Lie algebra section $\mathfrak g_p^s \to \mathfrak g_p^{lin}$. 
\end{enumerate}

Theorem \ref{thm:formal} specializes therefore to yield the following corollary: 

\begin{corollary} \label{coro:Cerveau}(Dominique Cerveau)
Let $\mathfrak g^s$ be the semi-simple part of the isotropy Lie algebra of  $\mathcal F $ at $\mathfrak g $. Then there exists a Lie algebra morphism  $s:\mathfrak g^s\to \widehat{\mathcal F}_{p}$ and a formal Euler-like vector field $\mathcal E$ with respect to which the image of $s$ is made of formally linear vector fields.
\end{corollary}

A comparison of this Corollary with Theorem 2.1 in Dominique Cerveau's \cite{Cerveau} shows that both statements are equivalent (although stated and proved quite differently here). Also, for $L=\{p\}$, Corollary \ref{coro:holonomyIsSemiSimple} recovers the second part of Theorem 2.2 in \cite{Cerveau}.

\subsubsection{Levi theorem for projective foliations}

Let us assume that  $\mathcal F$ is a projective module over $C^\infty(M)$ 
\footnote{i.e. ``Debord foliations'' in the terminology of \cite{LLS}.}.
In this case \cite{Debord}, there exists a Lie algebroid $(A, [\cdot, \cdot], \rho)$, such that the anchor map $\rho\colon A \to \coprod_{m \in M} T_m {\mathcal F} $, although it is not an isomorphism at every point,  is an isomorphism (of  $C^\infty(M)$-modules)  at the level of sections:
 \begin{align*}
     \rho\colon \Gamma(A) \cong   {\mathcal F}.
 \end{align*}
 Since all vector fields in $ \mathcal F$ vanish at $p$, we have that $\rho|_p =0 $, so that the fiber  of $A_p$ is a Lie algebra: it is easily shown to coincide with the isotropy Lie algebra ${\mathfrak g}_{p}$. Applying Corollary \ref{coro:Cerveau} to this situation yields the following result, where $ \widehat{\Gamma}(A) $ stands for formal sections of a vector bundle $A$ near  $p$:
 
\begin{corollary} \label{coro:WDZ} (\cite{MR1800493,MR1881647,Zung})
Let$ A$ be the Lie algebroid  associated to a projective singular foliation made of vector fields vanishing at  $ \{p\}$.
Denote by $A_p^s$  the semi-simple part of the isotropy Lie algebra $A_p $. Then there exists a Lie algebra morphism  $s:A_p^s\to \widehat{\Gamma}(A) $ and a formal Euler-like vector field $\mathcal E$ with respect to which the image of $\rho \circ s$ is made of formally linear vector fields.
\end{corollary}

This statement indeed holds true for any Lie algebroid, see \cite{MR1800493,MR1881647,Zung}.

\subsection{Sections to the Holonomy Lie ($\infty$-) algebroid}

Let $\mathcal F$ be a locally real analytic singular foliation. For every leaf $L$ such that $\pi_1(L)=\pi_2(L)=0$, Corollary \ref{cor:sectohol} assures the existence of a Lie algebroid section $A_L^s\to A_L$. Using Theorem \ref{thm:formal}, we can loosen the 2-connectedness condition for $M$ as follows:

\begin{proposition}\label{prop:sectoAL}
 Let $\mathcal F$ be a locally real analytic singular foliation and $L$ a simply connected and locally closed leaf, such that there exists a Lie algebroid section $A_L^s\to A_L^{lin}$. Then there exists a Lie algebroid section $A_L^s\to A_L$.
\end{proposition}
\begin{proof}
Let $c$ be the Artin-Rees bound for $\mathcal F$ at $L$. By ``stopping early'' in the iteration for Theorem \ref{thm:formal}, we obtain a section $s=s^{c+1}:\Gamma(A_L^s)\to \mathcal F$ and an Euler-like vector field $\mathcal E=\mathcal E^{c+1}$ such that $[s (\xi), s(\zeta)]- s([\xi,\zeta]) \in  I^{c+1}_L\mathfrak X^v\cap \mathcal F\subset I_L\mathcal F$ for all $\xi,\zeta \in \Gamma(A_L^s)$. Such a map $s$ induces a section $A_L^s\to A_L$ which is a Lie algebroid section.
\end{proof}

Proposition \ref{prop:sectoAL} can be generalized as follows. For the sake of simplicity, we will assume below that the formal $A_L^s$-action in Theorem \ref{thm:formal} is convergent, and that the leaf $L$ is compact, so that we may refer to existing results in \cite{LLS} and \cite{LGR}. These additional assumptions are certainly not relevant for both Propositions below, but avoiding them would require to extend to the formal setting the statements we will refer to.
In \cite{LLS}, it is shown that every real analytic singular foliation $\mathcal F $ is, locally on a neighbourhood $U$ of a point, the image through the anchor map of a \emph{universal Lie $\infty$-algebroid}, i.e.  a Lie $\infty $-algebroid $\mathbb U^{\mathcal F}=(E_{-i},[\cdots]_i, \rho)$ whose  $1 $-ary bracket $d=[\cdot]_1$, together with its anchor map:
 $$  \cdots  \overset{d}{\longrightarrow}  \Gamma(E_{-2})  \overset{d}{\longrightarrow}   \Gamma(E_{-1})  \overset{\rho}{\longrightarrow} \mathcal F |_U$$
form a projective resolution of $\mathcal F $. 
In Theorem 2.26 in \cite{LGR}, the universal Lie $\infty$-algebroid is shown to exists in a neighborhood of a compact leaf. The restriction of such a Lie $\infty $-algebroid $ \mathbb U^\mathcal F $ to  $L $ yields a transitive Lie $\infty $-algebroid  over $L $ denoted by $\mathbb U^\mathcal F |_L $. It admits a canonical Lie $ \infty$-morphism onto $A_L $. We call $ \Pi$ its composition with the projection $A_L \to  A_L^s$.

\begin{proposition}\label{prop:sectoU}
 Let $\mathcal F$ be a locally real analytic singular foliation and $L$ a simply connected and compact leaf, such that there exists a Lie algebroid section $A_L^s\to A_L^{lin}$. 
 We assume that the formal section $ \Gamma(A_L^s) \to \mathcal F  $ whose existence is granted by Theorem \ref{thm:formal} can be chosen to converge in a neighborhood of $ L$.  Then $ \Pi$ admits a Lie $\infty $-algebroid section $A_L^s\to {\mathbb U}^\mathcal F {|_L} $.
\end{proposition}
\begin{proof}
The Lie algebroid action of $A_L^s $
 on $U $ defines a sub-foliation $A_L^s $ in $\mathcal F $, namely the image through the anchor map  of the transformation Lie algebroid of this action. 
 In view of Theorem 2.9 in \cite{LLS}, there exists a Lie $ \infty$-algebroid morphism $ \Phi$ from this transformation Lie algebroid  to the universal Lie $ \infty$-algebroid $\mathbb U^\mathcal F$. The desired morphism is the restriction of $ \Phi$ to the leaf $L$.
 \end{proof}

An important question for a given singular foliation is to know whether or not it comes from a Lie algebroid action \cite{AZ13}. When the leaf $L$ is a point $ \{p\}$, the rank of such a Lie algebroid has to be greater or equal to the dimension  of the isotropy Lie algebra ${\mathfrak g}_p$.  Although the general problem remains open,  ${\mathfrak g}_p$ carries a Chevalley-Eilenberg cohomology $3$-class, called the NMRLA-class, that obstructs the possibility to have a Lie algebroid whose rank is minimal i.e. equal to ${\mathrm{dim}}( {\mathfrak g}_p) $ (see Proposition~4.29 in \cite{LLS}). Proposition \ref{prop:sectoU} has strong implications for this class: it shows that it is effa\c{c}able.\\

Recall that for $\mathfrak g $ a Lie algebra and $V$ a finite dimensional $\mathfrak g $-module, a class $\omega $ in a Chevalley-Eilenberg cohomology group $H^k(\mathfrak g, V) $ is \emph{effa\c{c}able} (or erasable) if there exists a finite dimensional $\mathfrak g $-module $W$ containing $V$ such that the image of $\omega  $ in $ H^k(\mathfrak g, W)$ is zero.\\

Let us briefly describe the NMRLA class assuming $L = \{p\}$ is a point leaf. In this case, $\mathbb U^\mathcal F|_{\{p\}} $ is a Lie $ \infty$-algebra whose $1$-ary bracket can be assumed to be zero. Then, its degree $(-1)$ component is a Lie algebra isomorphic to ${\mathfrak{g}}_p$ (see Proposition 4.14in \cite{LLS}), its degree $(-2)$-component is a ${\mathfrak{g}}_p $-module $V$, and the restriction to ${\mathfrak{g}}_p $ of the $3$-ary bracket is a Chevalley-Eilenberg $3$-cocycle valued in $V$ (see Proposition 4.27 in \cite{LLS}), defining the NMRLA class. 

\begin{proposition}
Let $\mathcal F $ be a locally real analytic singular foliation and $\{p\}$ a point leaf such that the formal section $\mathfrak g_p^s \to \mathcal F  $ whose existence is granted by Corollary \ref{coro:Cerveau} can be chosen to converge in a neighborhood of $ p$. 
Then the NMRLA-class of $\mathcal F $ at $p$ is effa\c{c}able.
\end{proposition}
\begin{proof}
In view of Theorem 1 in \cite{Hochschild}, a cohomology class is effa\c{c}able if and only if its restriction to a maximal semi-simple Lie subalgebra is zero. Let $ \Phi \colon \mathfrak g_p^s \to \mathbb U^\mathcal F |_p $ be  a Lie $\infty $-algebroid morphism as in Proposition \ref{prop:sectoU}. 
The Taylor coefficient $ \Phi_1 \colon  \mathfrak g_p^s  \to  \mathfrak g_p^s$ of  $ \Phi$  is the identity map and the second Taylor coefficient $ \Phi_2 \colon \wedge^2  \mathfrak g_p^s  \to V$ satisfies (see Equation (4.10) in \cite{LLS})
  for all $a,b,c \in \mathfrak g_p^s$, $$
	\big\{a,b,c\big\}_3  = 	\big\{ a , \Phi_2 (b,c) \big\}_2 - \Phi_2\big(\{a,b\}_2,c\big)    + \circlearrowright{\hbox{\tiny{$abc$}}}  .$$
   This means that the restriction of the $3$-ary bracket to  $\mathfrak g_p^s $  is a Chevalley-Eilenberg cocycle. This concludes the proof. 
   \end{proof}

\subsection{Transversally quadratic simply connected leaves}

Let $\mathcal{F}$ be a locally real analytic foliation and $L$ a leaf. We say that a leaf $L $ is  \emph{transversally quadratic} if its transverse singular foliation (see Lemma \ref{lem:splitting}) is made of vector fields vanishing at least quadratically.
There is an easy characterization in terms of the holonomy Lie algebroid of the leaf $L$:

\begin{proposition}
\label{prop:transQuadr}
A leaf $L$ is transversally quadratic if and only if $A_L^{lin}=A_L^s= TL$. In particular, the normal bundle $\nu$ carries a natural flat connection $ \nabla^\nu$.
\end{proposition}
\begin{proof}
By definition of $A_L^{lin} $, the first part of the proposition follows from the following intermediate characterization of transversally quadratic leaves: A leaf is transversally quadratic if $\mathcal F^v\subset I^2_L\mathfrak X^v$. The second part of the proposition follows from the existence, for every leaf $L$, of a natural $A_L^{lin} $-action on $\nu $, see Subsection~\ref{sec:linhol}.
\end{proof}

\begin{remark}
It follows immediately from Proposition \ref{prop:transQuadr} that a leaf $L \subset M$ whose normal bundle is not flat can not be transversally quadratic, which is a very strong constraint. For instance, $S^2 \subset TS^2$ can not be transversally quadratic.
\end{remark}

For a regular foliation, it is well-known that in a neighborhood of a simply connected leaf $L$, the  foliation is ``trivial'', i.e. formally, it is isomorphic to the direct product of the leaf $L$ with an open disk. 
The same phenomena occurs for transversally quadratic leaves: 

\begin{theorem}
\label{theo:transvQuadr}
Every simply-connected, transversally quadratic and locally closed leaf $L$ of a locally real analytic singular foliation $ \mathcal F$ is formally trivial, i.e. the formal jet $\widehat{\mathcal F} $ along $L$ is  isomorphic to the direct product\footnote{Direct products of $L$ with a singular foliation are discussed in Example \ref{ex:directprod}.} of $L$ with the formal jet of the transverse foliation.
\end{theorem}
\begin{proof}
Both conditions in Theorem \ref{thm:formal2} are satisfied: $L$ is simply connected by assumption and a section $A_L^s \to A_L^{lin} $ exists since both algebroids coincide with $TL$ by  Proposition \ref{prop:transQuadr}.
There is therefore a formal isomorphism between $ \widehat{\mathcal F}$ and 
$ TL \widehat \ltimes \widehat{\mathcal R}$, with $ \mathcal R$ the radical foliation.
In this case, however, there are several obvious identifications:
\begin{enumerate}
    \item The radical foliation $\mathcal R$ of $ \mathcal F$ is simply the transverse singular foliation.
    \item By Proposition \ref{prop:transQuadr},  the normal bundle $ \nu$ is flat. Since $L$ is simply connected, it is indeed a trivial vector bundle: $ \nu  \simeq  L \times \nu_p $, with $\nu_p $ some given fiber. 
\end{enumerate}
The semi-direct product is then reduced to a direct product. This gives the desired formal isomorphism.
\end{proof}

\begin{remark}\label{rmk:nooidnofish}
 Theorem \ref{theo:transvQuadr} is a purely singular foliation phenomenon: there is no such a result for Lie algebroids or Poisson structures. In fact, even for regular Poisson or Lie algebroid structures there is no such a result.
 For instance, choose of a volume form $ \omega$ on the $2$-sphere $S^2$, let $ \pi = \omega^{-1}$ be its inverse Poisson structure and consider the Poisson structure on $ S^2 \times \mathbb R $ given by $  e^t \pi \oplus 0$ with $t$ the parameter on $ \mathbb R$. The symplectic leaves are the fibers of the projection $  S^2 \times \mathbb R  \to \mathbb R   $. They are therefore simply-connected and their transverse Poisson structure is zero (in particular, it is transversally quadratic: it vanishes at order at least $2$). But since the volumes of all the symplectic leaves are different, this Poisson structure can not be isomorphic to a direct product of $ \pi = \omega^{-1}$ with the trivial Poisson structure on $ \mathbb R$ in a neighborhood of a given leaf (even formally).  
 
 Similarly, consider sections of the vector bundle $ A = T(S^2 \oplus \mathbb R)  $ over the manifold $ S^2\times \mathbb R $ as pairs $(X,f) $ or $ (Y,g)$ with $X,Y$ being $t$-dependent vector fields tangent to $S^2 $ and $f,g$ $t$-dependent real-valued functions on $S^2 $ (with $t$ the parameter along $ \mathbb R$.). The bracket:
  $$  [ (X, f)  , (Y , g) ] = \left([X,Y] , X[g] - Y[f] + t \omega(X,Y) \right)$$
is a Lie algebroid bracket on $A$. The leaves of $A$ are 2-spheres: they are therefore simply connected. The transverse Lie algebroid $ T \mathbb R \to \mathbb R$ has trivial anchor and trivial bracket. The restriction of the Lie algebroid $A$ to any two leaves are isomorphic, except for the exceptional leaf $t=0$.  Hence the Lie algebroid $A$ is not a direct product near the leaf $t=0$.  
\end{remark}

By applying the ``stopping early'' strategy from Proposition \ref{prop:sectoAL} in the proof of the previous Theorem, we obtain the following result:

\begin{corollary}
The holonomy Lie algebroid $A_L $ of a simply connected, transversally quadratic and locally closed leaf $L$ is the direct sum of $TL$ with the isotropy Lie algebra of its transverse foliation.
\end{corollary}
\begin{proof}
By Proposition \ref{prop:sectoAL}, a Lie algebroid section 
 $\underline{s} :  \mathfrak X(L)  \to   \mathcal F^{v}/ I_L  \mathcal F^{v}  \simeq \Gamma(A_L)$ exists.
 This section makes the isotropy Lie algebra bundle $ {\mathrm{ker}}(\rho)$ of $A_L$ a flat Lie algebra bundle. Since $L$ is simply connected, it is a trivial Lie algebra bundle.
\end{proof}

The proof of Theorem \ref{theo:transvQuadr} is fact shows the following more general statement: 

\begin{theorem}\label{thm:TLtoAlinimpliestriviality}
A simply-connected and locally closed leaf $L$ of a locally real analytic singular foliation $ \mathcal F$ is formally trivial (i.e. the formal jet $\widehat{\mathcal F} $ along $L$ is  isomorphic to the direct product of $L$ with the formal jet of the transverse foliation) if and only if there exists a Lie algebroid section $ TL \to A^{lin}_L$. 
\end{theorem}

{\small
\bibliographystyle{alpha}
\bibliography{biblio}
}

\end{document}